\documentclass[a4paper,10pt]{amsart}
\usepackage[colorlinks,linkcolor=blue,citecolor=blue]{hyperref}
\usepackage{latexsym, amssymb, amsmath, amscd, amsthm, mathrsfs, bbm}
\usepackage[all, knot]{xy}
\xyoption{arc}

\usepackage{anysize}\marginsize{20mm}{20mm}{25mm}{25mm}
\allowdisplaybreaks[4]

\def \To{\longrightarrow}
\def \Hom{\operatorname{Hom}}
\def \Shuffle{\operatorname{Shuffle}}
\def \Vec{\operatorname{Vec}}
\def \Ker{\operatorname{Ker}}
\def \H{\operatorname{H}}
\def \id{\operatorname{id}}

\def \op{\operatorname{op}}
\def \R{\mathcal{R}}
\def \E{\mathcal{E}}

\def \k{\mathbbm{k}}
\def \o{\omega}

\numberwithin{equation}{section}

\newtheorem{theorem}{Theorem}[section]
\newtheorem{lemma}[theorem]{Lemma}
\newtheorem{proposition}[theorem]{Proposition}
\newtheorem{corollary}[theorem]{Corollary}

\newtheorem{remark}[theorem]{Remark}

\begin{document}

\title[Cocycle formulas on finite abelian groups with applications]{EXPLICIT COCYCLE FORMULAS ON FINITE ABELIAN GROUPS with \\ APPLICATIONS to braided linear Gr-categories and \\ Dijkgraaf-Witten invariants$^\dag$}\thanks{\footnotesize $^\dag$Supported by NSFC 11431010, 11471186, 11571199 and 11571329.}



\author{Hua-Lin Huang, Zheyan Wan* and Yu Ye} 
\thanks{*Corresponding author.}

\address{Fujian Province University Key Laboratory of Computational Science, School of Mathematical Sciences, Huaqiao University, Quanzhou 362021, China} 
\email{hualin.huang@hqu.edu.cn}

\address{School of Mathematical Sciences, University of Science and Technology of China, Hefei 230026, China} \email{wanzy@mail.ustc.edu.cn}

\address{School of Mathematical Sciences, University of Science and Technology of China, Hefei 230026, China} 
\email{yeyu@ustc.edu.cn}

\date{}
\maketitle

\begin{abstract}
We provide explicit and unified formulas for the cocycles of all degrees on the normalized bar resolutions of finite abelian groups. This is achieved by constructing a chain map from the normalized bar resolution to a Koszul-like resolution for any given finite abelian group. With a help of the obtained cocycle formulas, we determine all the braided linear Gr-categories and compute the Dijkgraaf-Witten Invariants of the $n$-torus for all $n.$
\end{abstract}

\section{Introduction}
Throughout, let $\k$ be an algebraically closed field of characteristic zero and let $\k^*$ denote the multiplicative group $\k-\{0\}.$ Unless otherwise specified, all algebraic structures and linear operations are over $\k.$ Our main aim is to provide explicit and unified formulas for the cocycles on the normalized bar resolutions (normalized cocycles) of finite abelian groups. Some applications to braided linear Gr-categories and Dijkgraaf-Witten Invariants (DW invariant) are also considered.

The cohomology groups of finite abelian groups are computable thanks to the well known Lyndon-Hochschild-Serre spectral sequence \cite{hs,lyndon}. However, the explicit formulas of normalized cocycles are not clear in literatures. Such explicit formulas of normalized cocycles, instead of the cohomology groups, are necessary in many respects of mathematics and physics. Besides the connections to braided linear Gr-categories and DW invariants involved in the present paper, normalized 2-cocycles are necessary in projective representation theory of finite groups \cite{frucht, kar}; normalized 3-cocycles are indispensable in the classification program of pointed finite tensor categories and quasi-quantum groups \cite{eg1, eg2, g, qha6, bgrc1, qha3};  normalized cocycles of all degrees are very important in the theory of symmetry protected topological orders \cite{cw, cw2, spt}. 

Our approach of formulating the normalized cocycles is straightforward and elementary. First we construct a Koszul-like resolution of a finite abelian group $G$ by tensoring the minimal resolutions of cyclic factors of $G$ and give a complete set of representatives of cocycles for this resolution. Then we construct a chain map from the normalized bar resolution to this Koszul-like resolution. Finally we get the desired explicit and unified formulas of normalized cocycles on $G$ by pulling back those on the Koszul-like resolution along the chain map. We remark that, in principle, the method of Lyndon-Hochschild-Serre spectral sequence may also help one formulate explicit forms of normalized cocycles with nearly as much effort as we need here. 

Here is a brief description of the content. In Section 2, we provide formulas of normalized cocycles of all degrees on any finite abelian groups. In Section 3, we use the formula of normalized 3-cocycles to determine the braided monoidal structures on linear Gr-categories.
In Section 4, we give a formula for the Dijkgraaf-Witten invariant of the $n$-torus for all $n$ and obtain the dimension formula for irreducible projective representations of an arbitrary finite abelian group.

\section{Explicit formulas of normalized cocycles on finite abelian groups}
In this section, we use freely the concepts and notations about group cohomology in the book \cite{w} of Weibel. Let $G$ be a group and $(B_{\bullet},\partial_{\bullet})$ be its normalized bar resolution. Applying $\Hom_{\mathbb{Z}G}(-,\k^{\ast})$ one gets a complex $(B^{\ast}_{\bullet},\partial^{\ast}_{\bullet}).$ Denote the group of normalized $n$-cocycles by $Z^n(G,\k^*),$ which is $\Ker \partial^*_n.$ In general, it is hard to determine $Z^n(G,\k^*)$ directly as the normalized bar resolution is far too large. Our strategy of overcoming this is to get first a simpler resolution of $G$ whose cocycles are easy to compute and then construct a chain map from the normalized bar resolution to it which will help to determine $Z^n(G,\k^*)$ eventually.

\subsection{A Koszul-like resolution} 
From now on let $G$ be a finite abelian group. Write $G=\mathbb{Z}_{m_{1}}\times\cdots \times\mathbb{Z}_{m_{n}}$ where $m_i|m_{i+1}$ for $1\le i\le n-1$ and for every $\mathbb{Z}_{m_{i}}$ fix a generator $g_{i}$ for $1\leq i\leq n.$ It is well known that the following periodic sequence is a free resolution of the
trivial $\mathbb{Z}_{m_{i}}$-module $\mathbb{Z}:$
\begin{equation}\label{resolution}
\cdots\longrightarrow \mathbb{Z}\mathbb{Z}_{m_{i}}\stackrel{T_{i}}\longrightarrow
\mathbb{Z}\mathbb{Z}_{m_{i}}\stackrel{N_{i}}\longrightarrow\mathbb{Z}\mathbb{Z}_{m_{i}}\stackrel{T_{i}}\longrightarrow
\mathbb{Z}\mathbb{Z}_{m_{i}}\stackrel{N_{i}}\longrightarrow
\mathbb{Z}\longrightarrow 0,\end{equation}
where $T_{i}=g_{i}-1$ and $N_{i}=\sum\limits_{j=0}^{m_{i}-1}g_{i}^{j}$. 

Consider the tensor product of the above periodic resolutions of the cyclic factors of $G.$ The resulting complex, denoted by 
$(K_{\bullet},d_{\bullet}),$ is as follows. For each sequence  $a_{1},\ldots,a_{n}$ of nonnegative integers, let $\Phi(a_{1},\ldots,a_{n})$ be a free
generator in degree $a_{1}+\cdots+a_{n}$. Thus
$$K_{m}:=\bigoplus_{a_{1}+\cdots+a_{n}=m} (\mathbb{Z}G)\Phi(a_{1},\ldots,a_{n}).$$ For all $1\leq i\leq n,$ define
$$d_{i}(\Phi(a_{1},\ldots,a_{n}))=\left \{
\begin{array}{lll} 0, &\;\;\;\;a_{i}=0;
\\ (-1)^{\sum_{l<i}a_{l}}N_{i}\Phi(a_{1},\ldots,a_{i}-1,\ldots,a_{n}), & \;\;\;\;0\neq a_{i}\;\textrm{even;}
\\(-1)^{\sum_{l<i}a_{l}}T_{i}\Phi(a_{1},\ldots,a_{i}-1,\ldots,a_{n}), &
\;\;\;\;0\neq a_{i}\;\textrm{odd.}
\end{array} \right.$$
The differential $d$ is set to be $d_{1}+\cdots +d_{n}$. Then 
$(K_{\bullet},d_{\bullet})$ is a free resolution of the trivial $\mathbb{Z}G$-module $\mathbb{Z}$. The main goal of this subsection is to determine the explicit cocycles of this Koszul-like resolution.

For the convenience of the exposition, we fix some notations before moving on. For any $1\le r_1<\cdots<r_l\le n$, define $\Phi_{r_1^{\lambda_1}\cdots r_l^{\lambda_l}}:=\Phi(0,\dots,\lambda_1,\dots,\lambda_l,\dots,0)$ where $\lambda_i\ge1$ lies in the $r_i$-th position. If $\lambda_i=1$ for some $1\le i\le l$, sometimes we drop it for brevity.
It is clear that any cochain $f\in\Hom_{\mathbb{Z}G}(K_k,\k^*)$ is uniquely determined by its values on $\Phi_{r_1^{\lambda_1}\cdots r_l^{\lambda_l}}$. Write $f_{r_1^{\lambda_1}\cdots r_l^{\lambda_l}}=f(\Phi_{r_1^{\lambda_1}\cdots r_l^{\lambda_l}})$.

\begin{theorem}\label{complete}
The following \begin{equation} \label{ck}  \left\{ f\in \Hom_{\mathbb{Z}G}(K_k,\k^*)
\left|
  \begin{array}{llll}
    f_{r_1^{\lambda_1}\cdots r_l^{\lambda_l}}=1\emph{ if }\lambda_1\emph{ is even}, \\
   f_{r_1^{\lambda_1}\cdots r_l^{\lambda_l}}=\zeta_{m_{r_1}}^{a_{r_1^{\lambda_1}\cdots r_l^{\lambda_l}}}\emph{ if }\lambda_1\emph{ is odd}\\
     \emph{and}  \ 0\leq a_{r_1^{\lambda_1}\cdots r_l^{\lambda_l}}<m_{r_1} \  \emph{for} \ 1\leq r_1<\cdots<r_l\leq n\\
     \emph{where}\ \lambda_1+\cdots+\lambda_l=k,\ \lambda_i\ge1\ \emph{for}\ 1\le i\le l\\
      \end{array}
\right. \right\} \end{equation}
makes a complete set of representatives of $k$-cocycles of the complex $(K_{\bullet}^*,d_{\bullet}^*)$. 
\end{theorem}

\begin{proof}
Suppose $f\in \Hom_{\mathbb{Z}G}(K_k,\k^*)$ is a $k$-cocycle. We will show that $f$ is cohomologous to one in \eqref{ck}. Let $g\in \Hom_{\mathbb{Z}G}(K_{k-1},\k^*)$ be a $(k-1)$-cochain given by $g_{r_1^{\mu_1}\cdots r_l^{\mu_l}}=1$ if $\mu_1$ is even and 
$g_{r_1^{\mu_1}\cdots r_l^{\mu_l}}=(f_{r_1^{\mu_1+1}\cdots r_l^{\mu_l}})^{\frac{1}{m_{r_1}}}$ if $\mu_1$ is odd. Consider $f'=f-d^*g.$ Then clearly $f'_{r_1^{\lambda_1}\cdots r_l^{\lambda_l}}=1$ if $\lambda_1$ is even.
If $\lambda_1$ is odd, then by the cocycle condition for $f'$ we have 
$$(f'_{r_1^{\lambda_1}\cdots r_l^{\lambda_l}})^{m_{r_1}}\prod_{\begin{array}{cc}2\le i\le l\\\lambda_i\text{ even}\end{array}}(f'_{r_1^{\lambda_1+1}\cdots r_i^{\lambda_i-1}\cdots r_l^{\lambda_l}})^{(-1)^{\sum\limits_{j<i}\lambda_j+1}m_{r_i}}=1.$$
Hence $(f'_{r_1^{\lambda_1}\cdots r_l^{\lambda_l}})^{m_{r_1}}=1$, so $f'_{r_1^{\lambda_1}\cdots r_l^{\lambda_l}}=\zeta_{m_{r_1}}^{a_{r_1^{\lambda_1}\cdots r_l^{\lambda_l}}}$ for some $0\le a_{r_1^{\lambda_1}\cdots r_l^{\lambda_l}}<m_{r_1}$.

Suppose that $f_1$ and $f_2$ are two cocycles in \eqref{ck} satisfying $f_1-f_2=d^*h$ for some $(k-1)$-cochain $h\in \Hom_{\mathbb{Z}G}(K_{k-1},\k^*).$ Similarly as above, after subtracting a $(k-1)$-coboundary from $h$, we can assume that $h_{r_1^{\mu_1}\cdots r_l^{\mu_l}}=1$ if $\mu_1$ is even.
If $\lambda_1$ is even, then 
\begin{eqnarray*}
(f'_1-f'_2)_{r_1^{\lambda_1}\cdots r_l^{\lambda_l}}&=&(h_{r_1^{\lambda_1-1}\cdots r_l^{\lambda_l}})^{m_{r_1}}\prod_{\begin{array}{cc}2\le i\le l\\\lambda_i\text{ even}\end{array}}(h_{r_1^{\lambda_1}\cdots r_i^{\lambda_i-1}\cdots r_l^{\lambda_l}})^{(-1)^{\sum\limits_{j<i}\lambda_j}m_{r_i}}\\
&=&(h_{r_1^{\lambda_1-1}\cdots r_l^{\lambda_l}})^{m_{r_1}}=1.
\end{eqnarray*}
If $\lambda_1$ is odd, then by the preceding equation and the condition $m_i|m_{i+1}$ for $1\le i\le n-1$ we have
$$(f_1-f_2)_{r_1^{\lambda_1}\cdots r_l^{\lambda_l}}=\prod_{\begin{array}{cc}2\le i\le l\\\lambda_i\text{ even}\end{array}}(h_{r_1^{\lambda_1}\cdots r_i^{\lambda_i-1}\cdots r_l^{\lambda_l}})^{(-1)^{\sum\limits_{j<i}\lambda_j}m_{r_i}}=1.$$
Hence $f_1=f_2$.
\end{proof}

\begin{corollary}
If $G=\mathbb{Z}_{m_{1}}\times\cdots \times\mathbb{Z}_{m_{n}}$ where $m_i|m_{i+1}$ for $1\le i\le n-1$, then
$$\H^k(G,\k^*)=\prod_{r=1}^n\mathbb{Z}_{m_r}^{\sum\limits_{j=1}^k(-1)^{k+j}\binom{n-r+j-1}{j-1}}.$$
\end{corollary}
\begin{proof}
By Theorem \ref{complete}, $\H^k(G,\k^*)=\prod_{r=1}^n\mathbb{Z}_{m_r}^{N_{k,r}}$ where 
\begin{eqnarray*}
N_{k,r}&=&\#\{(r_2,\dots,r_l,\lambda_1,\dots,\lambda_l)\in\mathbb{N}^{2l-1}|r<r_2<\cdots<r_l\le n,\lambda_1+\cdots+\lambda_l=k,\lambda_1\text{ odd}\}\\
&=&\sum_{l=1}^k\binom{n-r}{l-1}\#\{(\lambda_1,\dots,\lambda_l)\in\mathbb{N}^l|\lambda_1+\cdots+\lambda_l=k,\lambda_1\text{ odd}\}.
\end{eqnarray*}
Denote $s_{k,l}=\#\{(\lambda_1,\dots,\lambda_l)\in\mathbb{N}^l|\lambda_1+\cdots+\lambda_l=k,\lambda_1\text{ odd}\}$ and $t_{k,l}=\#\{(\lambda_1,\dots,\lambda_l)\in\mathbb{N}^l|\lambda_1+\cdots+\lambda_l=k,\lambda_1\text{ even}\}.$ Then $s_{k,l}=t_{k+1,l}$ and $s_{k,l}+t_{k,l}=\binom{k-1}{l-1}$.
Hence $$N_{k,r}+N_{k-1,r}=\sum_{l=1}^k\binom{n-r}{l-1}\binom{k-1}{l-1}=\binom{n-r+k-1}{k-1}.$$
Therefore, $N_{k,r}=\sum\limits_{j=1}^k(-1)^{k+j}\binom{n-r+j-1}{j-1}$.
\end{proof}

\subsection{A chain map from $(B_{\bullet},\partial_{\bullet})$ to $(K_{\bullet},d_{\bullet})$}
We need some more notations to present our chain map. For any positive integers $s$ and $r,$ let $[\frac{s}{r}]$ denote the integer part of $\frac{s}{r}$ and let $s_r'$ denote the remainder of the division of $s$ by $r.$ When there is no risk of confusion, we omit the subscript in $s_r'.$ It is easy to observe that 
\begin{equation} \label{f}
 [\frac{s+t_r'}{r}]=[\frac{s+t-[\frac{t}{r}]r}{r}]=[\frac{s+t}{r}]-[\frac{t}{r}]
\end{equation}
for any three natural numbers $s, \ t $ and $r.$ We need the following technical lemma for later discussions.

\begin{lemma}\label{tl}
Let $r$ be a positive integer. For any $2l+1$ natural numbers $a_1, \ a_2, \dots, a_{2l+1},$ we have the following equation
\begin{eqnarray*}
 &&\sum_{i=1}^l [\frac{a_{2l+1}+a_{2l}}{r}] \cdots [\frac{a_{2i+3}+a_{2i+2}}{r}]  \left( [\frac{a_{2i+1}+(a_{2i}+a_{2i-1})'}{r}] - [\frac{(a_{2i+1}+a_{2i})'+a_{2i-1}}{r}] \right) \cdot \\
 && \qquad [\frac{a_{2i-2}+a_{2i-3}}{r}] \cdots [\frac{a_{2}+a_{1}}{r}]  \\
 &=& [\frac{a_{2l+1}+a_{2l}}{r}] \cdots [\frac{a_{3}+a_{2}}{r}] - [\frac{a_{2l}+a_{2l-1}}{r}] \cdots [\frac{a_{2}+a_{1}}{r}] .
 \end{eqnarray*}
\end{lemma}
\begin{proof}
By \eqref{f}, we have \[ [\frac{a_{2i+1}+(a_{2i}+a_{2i-1})'}{r}] - [\frac{(a_{2i+1}+a_{2i})'+a_{2i-1}}{r}] = [\frac{a_{2i+1}+a_{2i}}{r}] - [\frac{a_{2i}+a_{2i-1}}{r}].  \]
Then the lemma follows by an obvious elimination of consecutive terms. 
\end{proof}

Now we are ready to give a chain map from the normalized bar resolution $(B_{\bullet},\partial_{\bullet})$ to the Koszul-like resolution $(K_{\bullet},d_{\bullet}).$ Recall that $B_{m}$ is the free $\mathbb{Z}G$-module on the set of all symbols
$[h_{1},\ldots,h_{m}]$ with $h_{i}\in G$ and $m\geq 1$. In case $m=0$, the symbol $[\; ]$ denote $1\in \mathbb{Z}G$ and the map $\partial_{0}=\epsilon:\;B_{0}\to \mathbb{Z}$ sends
$[\; ]$ to $1$. For a more concise formulation, denote $(g_i)_r:=\sum\limits_{j=0}^{r-1}g_i^j$ for $1\leq i\leq n$ in the following.

The first four terms of the chain map, which will be used for later applications, are as follows:
\begin{eqnarray*}
F_{1}: &&B_{1}\To K_{1}\\
&&[g_{1}^{i_{1}}\cdots g_{n}^{i_{n}}]\mapsto
\sum_{s=1}^{n}g_{1}^{i_{1}}\cdots g_{s-1}^{i_{s-1}}(g_{s})_{i_{s}}\Phi_{s};\\
F_{2}: &&B_{2}\To K_{2}\\
&&[g_{1}^{i_{1}}\cdots g_{n}^{i_{n}},g_{1}^{j_{1}}\cdots g_{n}^{j_{n}}]\mapsto
\sum_{s=1}^{n}g_{1}^{i_{1}+j_{1}}\cdots g_{s-1}^{i_{s-1}+j_{s-1}}[\frac{i_{s}+j_{s}}{m_{s}}]\Phi_{s^2}\\
&&-\sum_{1\leq s<t\leq n}g_{1}^{i_{1}}\cdots g_{t-1}^{i_{t-1}}g_{1}^{j_{1}}\cdots g_{s-1}^{j_{s-1}}(g_{s})_{j_{s}}(g_{t})_{i_{t}}\Phi_{st};\\
F_{3}: &&B_{3}\To K_{3}\\
&&[g_{1}^{i_{1}}\cdots g_{n}^{i_{n}},g_{1}^{j_{1}}\cdots g_{n}^{j_{n}},g_{1}^{k_{1}}\cdots g_{n}^{k_{n}}]\mapsto\\
&& \sum_{r=1}^{n}[\frac{j_{r}+k_{r}}{m_{r}}]g_{1}^{j_{1}+k_{1}}\cdots g_{r-1}^{j_{r-1}+k_{r-1}}g_{1}^{i_{1}}\cdots g_{r-1}^{i_{r-1}}(g_{r})_{i_{r}}\Phi_{r^3}+\\
&&\sum_{1\leq r<t\leq n}[\frac{j_{r}+k_{r}}{m_{r}}]g_{1}^{j_{1}+k_{1}}\cdots g_{r-1}^{j_{r-1}+k_{r-1}}g_{1}^{i_{1}}\cdots g_{t-1}^{i_{t-1}}(g_{t})_{i_{t}}\Phi_{r^2t}+\\
&&\sum_{1\leq r<t\leq n}[\frac{i_{t}+j_{t}}{m_{t}}]g_{1}^{i_{1}+j_{1}}\cdots g_{t-1}^{i_{t-1}+j_{t-1}}g_{1}^{k_{1}}\cdots g_{r-1}^{k_{r-1}}(g_{r})_{k_{r}}\Phi_{rt^2}-\\
&&\sum_{1\leq r<s<t\leq n}g_{1}^{i_{1}}\cdots g_{t-1}^{i_{t-1}}(g_{t})_{i_{t}}g_{1}^{j_{1}}\cdots g_{s-1}^{j_{s-1}}(g_{s})_{j_{s}}g_{1}^{k_{1}}\cdots g_{r-1}^{k_{r-1}}(g_{r})_{k_{r}}\Phi_{rst};\\
F_{4}: &&B_{4}\To K_{4}\\
&&[g_{1}^{i_{1}}\cdots g_{n}^{i_{n}},g_{1}^{j_{1}}\cdots g_{n}^{j_{n}},g_{1}^{k_{1}}\cdots g_{n}^{k_{n}},g_{1}^{l_{1}}\cdots g_{n}^{l_{n}}]\mapsto\\
&& \sum_{r=1}^{n}[\frac{k_{r}+l_{r}}{m_{r}}][\frac{i_{r}+j_{r}}{m_{r}}]g_{1}^{i_{1}+j_{1}+k_{1}+l_{1}}\cdots g_{r-1}^{i_{r-1}+j_{r-1}+k_{r-1}+l_{r-1}}\Phi_{r^4}+\\
&&\sum_{1\leq r<s\leq n}[\frac{k_{r}+l_{r}}{m_{r}}]g_{1}^{k_{1}+l_{1}}\cdots g_{r-1}^{k_{r-1}+l_{r-1}}[\frac{i_{s}+j_{s}}{m_{s}}]g_{1}^{i_{1}+j_{1}}\cdots g_{s-1}^{i_{s-1}+j_{s-1}}\Phi_{r^2s^2}-\\
&&\sum_{1\leq r<s\leq n}[\frac{j_{s}+k_{s}}{m_{s}}]g_{1}^{j_{1}+k_{1}}\cdots g_{s-1}^{j_{s-1}+k_{s-1}}g_{1}^{l_{1}}\cdots g_{r-1}^{l_{r-1}}(g_{r})_{l_{r}}g_{1}^{i_{1}}\cdots g_{s-1}^{i_{s-1}}(g_{s})_{i_{s}}\Phi_{rs^3}-\\
&&\sum_{1\leq r<s\leq n}[\frac{k_{r}+l_{r}}{m_{r}}]g_{1}^{k_{1}+l_{1}}\cdots g_{r-1}^{k_{r-1}+l_{r-1}}g_{1}^{j_{1}}\cdots g_{r-1}^{j_{r-1}}(g_{r})_{j_{r}}g_{1}^{i_{1}}\cdots g_{s-1}^{i_{s-1}}(g_{s})_{i_{s}}\Phi_{r^3s}-\\
&&\sum_{1\leq r<s<t\leq n}[\frac{k_{r}+l_{r}}{m_{r}}]g_{1}^{k_{1}+l_{1}}\cdots g_{r-1}^{k_{r-1}+l_{r-1}}g_{1}^{j_{1}}\cdots g_{s-1}^{j_{s-1}}(g_{s})_{j_{s}}g_{1}^{i_{1}}\cdots g_{t-1}^{i_{t-1}}(g_{t})_{i_{t}}\Phi_{r^2st}-\\
&&\sum_{1\leq r<s<t\leq n}[\frac{j_{s}+k_{s}}{m_{s}}]g_{1}^{j_{1}+k_{1}}\cdots g_{s-1}^{j_{s-1}+k_{s-1}}g_{1}^{l_{1}}\cdots g_{r-1}^{l_{r-1}}(g_{r})_{l_{r}}g_{1}^{i_{1}}\cdots g_{t-1}^{i_{t-1}}(g_{t})_{i_{t}}\Phi_{rs^2t}-\\
&&\sum_{1\leq r<s<t\leq n}[\frac{i_{t}+j_{t}}{m_{t}}]g_{1}^{i_{1}+j_{1}}\cdots g_{t-1}^{i_{t-1}+j_{t-1}}g_{1}^{l_{1}}\cdots g_{r-1}^{l_{r-1}}(g_{r})_{l_{r}}g_{1}^{k_{1}}\cdots g_{s-1}^{k_{s-1}}(g_{s})_{k_{s}}\Phi_{rst^2}+\\
&&\sum_{1\leq r<s<t<u\leq n}g_{1}^{l_{1}}\cdots g_{r-1}^{l_{r-1}}(g_{r})_{l_{r}}g_{1}^{k_{1}}\cdots g_{s-1}^{k_{s-1}}(g_{s})_{k_{s}}g_{1}^{j_{1}}\cdots g_{t-1}^{j_{t-1}}(g_{t})_{j_{t}}g_{1}^{i_{1}}\cdots g_{u-1}^{i_{u-1}}(g_{u})_{i_{u}}\Phi_{rstu}
\end{eqnarray*}
for $0\leq i_{r},j_{r},k_{r},l_{r}< m_{r}$ and $1\leq r\leq n$.

In general, let $\alpha:=(\alpha_{11},\dots,\alpha_{1n},\dots,\alpha_{k1},\dots,\alpha_{kn})$ where each $\alpha_{ij} \in [0, m_j)$ and is viewed as an integer modulo $m_j$ for all $1\leq i\leq k.$ We also write $\alpha=(\alpha_1,\dots,\alpha_k)$ where $\alpha_u=(\alpha_{u1},\dots,\alpha_{un})$ for $1\leq u\leq k.$
For brevity, in the following we denote the group element $g_{1}^{\alpha_{i1}}\cdots g_{n}^{\alpha_{in}}$ by $g^{\alpha_i}.$ Given any $1\leq r\leq n$, $[a,b]\subseteq[1,k]$, $a,b\in\mathbb{N}$ and $\alpha$, denote 
$$\xi_{r,[a,b]}^{\alpha}:=\left\{
\begin{array}{ll}[\frac{\alpha_{br}+\alpha_{b-1,r}}{m_{r}}]\cdots[\frac{\alpha_{a+1,r}+\alpha_{ar}}{m_{r}}]g_{1}^{\alpha_{b1}+\cdots+\alpha_{a1}}\cdots g_{r-1}^{\alpha_{b,r-1}+\cdots+\alpha_{a,r-1}},&\;\;\;a-b\text{ odd;}\\

[\frac{\alpha_{br}+\alpha_{b-1,r}}{m_{r}}]\cdots[\frac{\alpha_{a+2,r}+\alpha_{a+1,r}}{m_{r}}]g_{1}^{\alpha_{b1}+\cdots+\alpha_{a,1}}\cdots g_{r-1}^{\alpha_{b,r-1}+\cdots+\alpha_{a,r-1}}(g_{r})_{\alpha_{ar}},&\;\;\;a-b\text{ even.}
\end{array}\right.$$
Define 
\begin{eqnarray}\label{Fkabelian}
&& F_{k} \colon B_{k}\To K_{k}  \\
&&[g^{\alpha_1},\dots,g^{\alpha_{k}}]\mapsto \sum_{l=1}^{k}\sum_{\begin{array}{ccc}1\leq r_{1}<\cdots<r_{l}\leq n\\\lambda_1+\cdots+\lambda_l=k\\\lambda_i\ge1\text{ for }1\le i\le l\end{array}}(-1)^{\sum\limits_{1\leq i<j\leq l}\lambda_{i}\lambda_{j}}\xi_{r_{1},[a_{1},b_{1}]}^{\alpha}\cdots\xi_{r_{l},[a_{l},b_{l}]}^{\alpha}\Phi_{r_1^{\lambda_1}\cdots r_l^{\lambda_l}}  \notag
\end{eqnarray}
where $a_u = {\sum\limits_{i=u+1}^l \lambda_i}+1$ and $b_u = {\sum\limits_{i=u}^l \lambda_i}$ for $1 \le u \le l.$ Clearly, the interval $[1,k]$ is the disjoint union of the $[a_i,b_i]$'s.

\begin{proposition} \label{chainmap}
The following diagram is commutative.
\begin{figure}[hbt]
\begin{picture}(150,50)(50,-40)
\put(0,0){\makebox(0,0){$ \cdots$}}\put(10,0){\vector(1,0){20}}\put(40,0){\makebox(0,0){$B_{3}$}}
\put(50,0){\vector(1,0){20}}\put(80,0){\makebox(0,0){$B_{2}$}}
\put(90,0){\vector(1,0){20}}\put(120,0){\makebox(0,0){$B_{1}$}}
\put(130,0){\vector(1,0){20}}\put(160,0){\makebox(0,0){$B_{0}$}}
\put(170,0){\vector(1,0){20}}\put(200,0){\makebox(0,0){$\mathbb{Z}$}}
\put(210,0){\vector(1,0){20}}\put(240,0){\makebox(0,0){$0$}}

\put(0,-40){\makebox(0,0){$ \cdots$}}\put(10,-40){\vector(1,0){20}}\put(40,-40){\makebox(0,0){$K_{3}$}}
\put(50,-40){\vector(1,0){20}}\put(80,-40){\makebox(0,0){$K_{2}$}}
\put(90,-40){\vector(1,0){20}}\put(120,-40){\makebox(0,0){$K_{1}$}}
\put(130,-40){\vector(1,0){20}}\put(160,-40){\makebox(0,0){$K_{0}$}}
\put(170,-40){\vector(1,0){20}}\put(200,-40){\makebox(0,0){$\mathbb{Z}$}}
\put(210,-40){\vector(1,0){20}}\put(240,-40){\makebox(0,0){$0$}}

\put(40,-10){\vector(0,-1){20}}
\put(80,-10){\vector(0,-1){20}}
\put(120,-10){\vector(0,-1){20}}
\put(158,-10){\line(0,-1){20}}\put(160,-10){\line(0,-1){20}}
\put(198,-10){\line(0,-1){20}}\put(200,-10){\line(0,-1){20}}

\put(60,5){\makebox(0,0){$\partial_{3}$}}
\put(100,5){\makebox(0,0){$\partial_{2}$}}
\put(140,5){\makebox(0,0){$\partial_{1}$}}

\put(60,-35){\makebox(0,0){$d$}}
\put(100,-35){\makebox(0,0){$d$}}
\put(140,-35){\makebox(0,0){$d$}}

\put(50,-20){\makebox(0,0){$F_{3}$}}
\put(90,-20){\makebox(0,0){$F_{2}$}}
\put(130,-20){\makebox(0,0){$F_{1}$}}

\end{picture}
\end{figure}
\end{proposition}

\begin{proof} 
We start with some conventions. Denote  
$$\E_{r^{\lambda}}:=\left\{
\begin{array}{ll}
N_r,&\;\;\;\lambda\text{ even;}\\

T_r,&\;\;\;\lambda\text{ odd.}
\end{array}\right.$$
Then 
\begin{eqnarray*}
d\Phi_{r_1^{\lambda_1}\cdots r_l^{\lambda_l}}&=&\E_{r_1^{\lambda_1}}\Phi_{r_1^{\lambda_1-1}\cdots r_l^{\lambda_l}}+
(-1)^{\lambda_1}\E_{r_2^{\lambda_2}}\Phi_{r_1^{\lambda_1}r_2^{\lambda_2-1}\cdots r_l^{\lambda_l}}  +\cdots+
(-1)^{\lambda_1+\cdots+\lambda_{l-1}}\E_{r_l^{\lambda_l}}\Phi_{r_1^{\lambda_1}\cdots r_l^{\lambda_l-1}} \ .
\end{eqnarray*}
For any given $\alpha=(\alpha_{11},\dots,\alpha_{1n},\dots,\alpha_{k1},\dots,\alpha_{kn}),$ let $\alpha'=(\alpha_2,\dots,\alpha_k), \ \alpha''=(\alpha_1,\dots,\alpha_{k-1})$ and $\alpha'_u=(\alpha_1,\dots,\alpha_{u-1},\alpha_u+\alpha_{u+1},\alpha_{u+2},\dots,\alpha_k), \quad \forall 1\leq u\leq k-1.$

With the above notations, $\partial_k([g^{\alpha_1}, \dots, g^{\alpha_k}])$ becomes 
\[
g^{\alpha_1}[g^{\alpha_2}, \dots, g^{\alpha_k}]+\sum_{u=1}^{k-1}(-1)^u [g^{\alpha_1},\dots, 
g^{\alpha_{u-1}}, g^{\alpha_u+\alpha_{u+1}}, g^{\alpha_{u+2}}, \dots, g^{\alpha_k}]+(-1)^k[g^{\alpha_1}, \dots, g^{\alpha_{k-1}}].
\]
Then the coefficient of $\Phi_{r_1^{\lambda_1}\cdots r_l^{\lambda_l}}$ in $F_{k-1}\partial_k([g^{\alpha_1}, \dots, g^{\alpha_k}])$ is 
\begin{equation} \label{c1}
(-1)^{\sum\limits_{1\leq i<j\leq l}\lambda_i\lambda_j} \left( g^{\alpha_1}\xi_{r_1,[a_1,b_1]}^{\alpha'}\cdots\xi_{r_l,[a_l,b_l]}^{\alpha'}+\sum_{u=1}^{k-1}(-1)^u\xi_{r_1,[a_1,b_1]}^{\alpha'_u}\cdots\xi_{r_l,[a_l,b_l]}^{\alpha'_u}
+(-1)^k\xi_{r_1,[a_1,b_1]}^{\alpha''}\cdots\xi_{r_l,[a_l,b_l]}^{\alpha''} \right)
\end{equation}
where $a_u = {\sum_{i=u+1}^l \lambda_i}+1$ and $b_u = {\sum_{i=u}^l \lambda_i}.$
For $1\leq r\leq n$, $[a,b]\subseteq[1,k-1]$, $a,b\in\mathbb{N}$ and $\alpha$, we have
$$ \xi_{r,[a,b]}^{\alpha'}=\xi_{r,[a+1,b+1]}^{\alpha}, \  \ \xi_{r,[a,b]}^{\alpha''}=\xi_{r,[a,b]}^{\alpha} \quad \text{and} \quad 
\xi_{r,[a,b]}^{\alpha'_u}=\left\{
\begin{array}{ll}
\xi_{r,[a+1,b+1]}^{\alpha},  \ \ \ &\text{ if } u<a;\\
\xi_{r,[a,b]}^{\alpha},  \ \ \ &\text{ if } u>b.
\end{array}\right.
$$
Hence
\begin{eqnarray*}
&&\sum_{u=1}^{k-1}(-1)^u\xi_{r_1,[a_1,b_1]}^{\alpha'_u}\cdots\xi_{r_l,[a_l,b_l]}^{\alpha'_u}\\
&=&\sum_{i=1}^l \sum_{u=a_i}^{b_i}(-1)^u \xi_{r_1,[a_1+1,b_1+1]}^{\alpha}\cdots\xi_{r_{i-1},[a_{i-1}+1,b_{i-1}+1]}^{\alpha} \xi_{r_i,[a_i,b_i]}^{\alpha'_u}\xi_{r_{i+1},[a_{i+1},b_{i+1}]}^{\alpha}\cdots\xi_{r_l,[a_l,b_l]}^{\alpha} \\
&=&\sum_{i=1}^l  \xi_{r_1,[a_1+1,b_1+1]}^{\alpha}\cdots\xi_{r_{i-1},[a_{i-1}+1,b_{i-1}+1]}^{\alpha} \xi_{r_{i+1},[a_{i+1},b_{i+1}]}^{\alpha}\cdots\xi_{r_l,[a_l,b_l]}^{\alpha} \sum_{u=a_i}^{b_i}(-1)^u \xi_{r_i,[a_i,b_i]}^{\alpha'_u}.
\end{eqnarray*}
Therefore we can rewrite \eqref{c1} as
\begin{eqnarray} \label{c1'}
&& (-1)^{\sum\limits_{1\leq i<j\leq l}\lambda_i\lambda_j} \left( g^{\alpha_1}\xi_{r_1,[a_1+1,b_1+1]}^{\alpha}\cdots\xi_{r_l,[a_l+1,b_l+1]}^{\alpha} +(-1)^k\xi_{r_1,[a_1,b_1]}^{\alpha}\cdots\xi_{r_l,[a_l,b_l]}^{\alpha} \right.  \\
&&\qquad + \sum_{i=1}^l  \xi_{r_1,[a_1+1,b_1+1]}^{\alpha}\cdots\xi_{r_{i-1},[a_{i-1}+1,b_{i-1}+1]}^{\alpha} \xi_{r_{i+1},[a_{i+1},b_{i+1}]}^{\alpha}\cdots\xi_{r_l,[a_l,b_l]}^{\alpha} \sum_{u=a_i}^{b_i}(-1)^u \left.  \xi_{r_i,[a_i,b_i]}^{\alpha'_u}  \right). \notag
\end{eqnarray}
It remains to compute $\sum\limits_{u=a_i}^{b_i}(-1)^u \xi_{r_i,[a_i,b_i]}^{\alpha'_u}.$ This is split into two cases according to the parity of $b_i-a_i.$ 

If $b_i-a_i$ is odd, then 
\begin{eqnarray} \label{odd}
&&\sum_{u=a_i}^{b_i}(-1)^u\xi_{r_i,[a_i,b_i]}^{\alpha'_u}\\
&=&\left( (-1)^{a_i}[\frac{\alpha_{b_i+1,r_i}+\alpha_{b_i,r_i}}{m_{r_i}}]\cdots[\frac{\alpha_{a_i+2,r_i}+(\alpha_{a_i,r_i}+\alpha_{a_i+1,r_i})'}{m_{r_i}}] \right.\notag\\
&&\quad +(-1)^{a_i+1}[\frac{\alpha_{b_i+1,r_i}+\alpha_{b_i,r_i}}{m_{r_i}}]\cdots[\frac{(\alpha_{a_i+1,r_i}+\alpha_{a_i+2,r_i})'+\alpha_{a_i,r_i}}{m_{r_i}}] \notag\\
&&\quad +\cdots \left. +(-1)^{b_i}[\frac{(\alpha_{b_i,r_i}+\alpha_{b_i+1,r_i})'+\alpha_{b_i-1,r_i}}{m_{r_i}}]\cdots[\frac{\alpha_{a_i+1,r_i}+\alpha_{a_i,r_i}}{m_{r_i}}]  \right) \cdot \notag \\ 
&& \qquad g_1^{\alpha_{b_i+1,1}+\cdots+\alpha_{a_i,1}}\cdots g_{r_i-1}^{\alpha_{b_i+1,r_i-1}+\cdots+\alpha_{a_i,r_i-1}} \notag \\
&\overset{Lemma \ref{tl}}{=}&\left( (-1)^{a_i}[\frac{\alpha_{b_i+1,r_i}+\alpha_{b_i,r_i}}{m_{r_i}}]\cdots[\frac{\alpha_{a_i+2,r_i}+\alpha_{a_i+1,r_i}}{m_{r_i}}] +  (-1)^{b_i}[\frac{\alpha_{b_i,r_i}+\alpha_{b_i-1,r_i}}{m_{r_i}}]\cdots [\frac{\alpha_{a_i+1,r_i}+\alpha_{a_i,r_i}}{m_{r_i}}] \right) \cdot \notag\\
&& \qquad g_1^{\alpha_{b_i+1,1}+\cdots+\alpha_{a_i,1}}\cdots g_{r_i-1}^{\alpha_{b_i+1,r_i-1}+\cdots+\alpha_{a_i,r_i-1}} \notag\\
&=&(-1)^{b_i}\xi_{r_i,[a_i,b_i]}^{\alpha}g_1^{\alpha_{b_i+1,1}}\cdots g_{r_i-1}^{\alpha_{b_i+1,r_i-1}}+(-1)^{a_i}\xi_{r_i,[a_i+1,b_i+1]}^{\alpha}g_1^{\alpha_{a_i,1}}\cdots g_{r_i-1}^{\alpha_{a_i,r_i-1}}. \notag
\end{eqnarray}

If $b_i-a_i$ is even, then similarly we have
\begin{eqnarray} \label{even}
&& \qquad \sum_{u=a_i}^{b_i}(-1)^u\xi_{r_i,[a_i,b_i]}^{\alpha'_u}\\
&=&(-1)^{a_i}[\frac{\alpha_{b_i+1,r_i}+\alpha_{b_i,r_i}}{m_{r_i}}]\cdots[\frac{\alpha_{a_i+3,r_i}+\alpha_{a_i+2,r_i}}{m_{r_i}}]g_1^{\alpha_{b_i+1,1}+\alpha_{a_i,1}}\cdots g_{r_i-1}^{\alpha_{b_i+1,r_i-1}+\cdots+\alpha_{a_i,r_i-1}}  (g_{r_i})_{(\alpha_{a_i,r_i}+\alpha_{a_i+1,r_i})'}  \notag\\
&&+\left( (-1)^{a_i+1}[\frac{\alpha_{b_i+1,r_i}+\alpha_{b_i,r_i}}{m_{r_i}}]
\cdots[\frac{\alpha_{a_i+3,r_i}+(\alpha_{a_i+1,r_i}+\alpha_{a_i+2,r_i})'}{m_{r_i}}] \right.  \notag\\
&&\qquad +(-1)^{a_i+2}[\frac{\alpha_{b_i+1,r_i}+\alpha_{b_i,r_i}}{m_{r_i}}]\cdots[\frac{(\alpha_{a_i+2,r_i}+\alpha_{a_i+3,r_i})'+\alpha_{a_i+1,r_i}}{m_{r_i}}] \notag\\
&&\qquad +\cdots \notag \\
&&\qquad \left. +(-1)^{b_i}[\frac{(\alpha_{b_i,r_i}+\alpha_{b_i+1,r_i})'+\alpha_{b_i-1,r_i}}{m_{r_i}}]\cdots[\frac{\alpha_{a_i+2,r_i}+\alpha_{a_i+1,r_i}}{m_{r_i}}] \right)\cdot \notag \\
&&\qquad \qquad g_1^{\alpha_{b_i+1,1}+\cdots+\alpha_{a_i,1}}\cdots g_{r_i-1}^{\alpha_{b_i+1,r_i-1}+\cdots+\alpha_{a_i,r_i-1}}(g_{r_i})_{\alpha_{a_i,r_i}}\notag\\
&=&(-1)^{a_i}[\frac{\alpha_{b_i+1,r_i}+\alpha_{b_i,r_i}}{m_{r_i}}]\cdots[\frac{\alpha_{a_i+3,r_i}+\alpha_{a_i+2,r_i}}{m_{r_i}}]g_1^{\alpha_{b_i+1,1}+\alpha_{a_i,1}}\cdots g_{r_i-1}^{\alpha_{b_i+1,r_i-1}+\cdots+\alpha_{a_i,r_i-1}}\cdot \notag\\
&&\qquad \left( (g_{r_i})_{\alpha_{a_i,r_i}+\alpha_{a_i+1,r_i}}-[\frac{\alpha_{a_i,r_i}+\alpha_{a_i+1,r_i}}{m_{r_i}}]N_{r_i} \right) \notag\\
&&+\left( (-1)^{a_i+1}[\frac{\alpha_{b_i+1,r_i}+\alpha_{b_i,r_i}}{m_{r_i}}]\cdots[\frac{\alpha_{a_i+3,r_i}+\alpha_{a_i+2,r_i}}{m_{r_i}}] +  (-1)^{b_i}[\frac{\alpha_{b_i,r_i}+\alpha_{b_i-1,r_i}}{m_{r_i}}]\cdots [\frac{\alpha_{a_i+2,r_i}+\alpha_{a_i+1,r_i}}{m_{r_i}}] \right) \cdot \notag\\
&&\qquad g_1^{\alpha_{b_i+1,1}+\cdots+\alpha_{a_i,1}}\cdots g_{r_i-1}^{\alpha_{b_i+1,r_i-1}+\cdots+\alpha_{a_i,r_i-1}}(g_{r_i})_{\alpha_{a_i,r_i}}\notag\\
&=&(-1)^{a_i}[\frac{\alpha_{b_i+1,r_i}+\alpha_{b_i,r_i}}{m_{r_i}}]\cdots[\frac{\alpha_{a_i+3,r_i}+\alpha_{a_i+2,r_i}}{m_{r_i}}]g_1^{\alpha_{b_i+1,1}+\alpha_{a_i,1}}\cdots g_{r_i-1}^{\alpha_{b_i+1,r_i-1}+\cdots+\alpha_{a_i,r_i-1}}\cdot \notag\\
&&\qquad \left( (g_{r_i})_{\alpha_{a_i,r_i}+\alpha_{a_i+1,r_i}}-(g_{r_i})_{\alpha_{a_i,r_i}}-[\frac{\alpha_{a_i,r_i}+\alpha_{a_i+1,r_i}}{m_{r_i}}]N_{r_i} \right)\notag\\
&&+  (-1)^{b_i}[\frac{\alpha_{b_i,r_i}+\alpha_{b_i-1,r_i}}{m_{r_i}}]\cdots [\frac{\alpha_{a_i+2,r_i}+\alpha_{a_i+1,r_i}}{m_{r_i}}] 
 g_1^{\alpha_{b_i+1,1}+\cdots+\alpha_{a_i,1}}\cdots g_{r_i-1}^{\alpha_{b_i+1,r_i-1}+\cdots+\alpha_{a_i,r_i-1}}(g_{r_i})_{\alpha_{a_i,r_i}}\notag\\
&=&(-1)^{a_i}\xi_{r_i,[a_i+1,b_i+1]}^{\alpha}g_1^{\alpha_{a_i,1}}\cdots g_{r_i}^{\alpha_{a_i,r_i}} + (-1)^{a_i+1}\xi_{r_i,[a_i,b_i+1]}^{\alpha} \E_{{r_i}^{b_i-a_i}} + (-1)^{b_i}\xi_{r_i,[a_i,b_i]}^{\alpha}g_1^{\alpha_{b_i+1,1}}\cdots g_{r_i-1}^{\alpha_{b_i+1,r_i-1}}. \notag
\end{eqnarray}

On the other hand, the term $\Phi_{r_1^{\lambda_1}\cdots r_l^{\lambda_l}}$ in $dF_k([g^{\alpha_1}, \dots, g^{\alpha_k}])$ comes from the differential of the terms 
\[ \Phi_{r_1^{\lambda_1+1}  r_2^{\lambda_2} \cdots r_l^{\lambda_l}}, \  \cdots, \  \Phi_{r_1^{\lambda_1} \cdots  r_{l-1}^{\lambda_{l-1}} r_l^{\lambda_l+1}}, \ \ \Phi_{s r_1^{\lambda_1}\cdots r_l^{\lambda_l}},\ \  \Phi_{r_1^{\lambda_1} s r_2^{\lambda_2}  \cdots r_l^{\lambda_l}}, \ \cdots, \  \Phi_{r_1^{\lambda_1} \cdots  r_{l-1}^{\lambda_{l-1}}  s r_l^{\lambda_l}}, \ \ \Phi_{r_1^{\lambda_1}\cdots r_l^{\lambda_l}s} \] 
in $F_k([g^{\alpha_1}, \dots, g^{\alpha_k}]).$ Therefore, its coefficient is 
\begin{eqnarray} \label{c2}
&&\sum_{1\leq u \leq l}(-1)^{\sum\limits_{1\leq i<j\leq l}\lambda_i\lambda_j+\sum\limits_{j \ne u} \lambda_j} \xi_{r_1,[a_1+1,b_1+1]}^{\alpha} \cdots \xi_{r_{u-1},[a_{u-1}+1,b_{u-1}+1]}^{\alpha} \xi_{r_u,[a_u,b_u+1]}^{\alpha} \cdot \\
&&\qquad \xi_{r_{u+1},[a_{u+1},b_{u+1}]}^{\alpha}  \cdots \xi_{r_l, [a_l, b_l]}^\alpha (-1)^{\lambda_1+\cdots+\lambda_{u-1}} \E_{r_u^{\lambda_u+1}} \notag \\
&+&\sum_{s=1}^{r_1-1}(-1)^{\sum\limits_{1\leq i<j\leq l}\lambda_i\lambda_j+\sum\limits_{i=1}^l \lambda_i} \xi_{r_1,[a_1,b_1]}^{\alpha} \cdots \xi_{r_l,[a_l, b_l]}^{\alpha}\xi_{s,[k,k]}^{\alpha}T_s\notag\\%
&+&\sum_{u=1}^{l-1} \sum_{s=r_u+1}^{r_{u+1}-1} (-1)^{\sum\limits_{1\leq i<j\leq l}\lambda_i\lambda_j+\sum\limits_{i=1}^l \lambda_i} \xi_{r_1,[a_1+1,b_1+1]}^{\alpha}\cdots \xi_{r_u,[a_u+1,b_u+1]}^{\alpha} \cdot \notag \\
&& \qquad \xi_{s, [a_u, a_u]} \xi_{r_{u+1},[a_{u+1},b_{u+1}]}^{\alpha} \cdots \xi_{r_l,[a_l,b_l]}^{\alpha}
 (-1)^{\lambda_1 + \cdots + \lambda_u} T_s \notag\\
&+&\sum_{s=r_l+1}^n(-1)^{\sum\limits_{1\leq i<j\leq l}\lambda_i\lambda_j+\sum\limits_{i=1}^l \lambda_i} \xi_{r_1,[a_1+1,b_1+1]}^{\alpha} \cdots \xi_{r_l,[a_l+1,b_l+1]}^{\alpha} \xi_{s,[1,1]}^{\alpha}(-1)^{\sum\limits_{i=1}^l\lambda_i}T_s. \notag 
\end{eqnarray}

Noting that $\xi_{s,[a,a]}^{\alpha}T_s=g_1^{\alpha_{a1}}\cdots g_{s-1}^{\alpha_{a,s-1}}(g_s^{\alpha_{a,s}}-1)$, then one has the following equations:
\begin{eqnarray*}
&&\sum_{s=1}^{r_1-1}(-1)^{\sum\limits_{1\leq i<j\leq l}\lambda_i\lambda_j+\sum\limits_{i=1}^l \lambda_i} \xi_{r_1,[a_1,b_1]}^{\alpha} \cdots \xi_{r_l,[a_l,b_l]}^{\alpha}\xi_{s,[k,k]}^{\alpha}T_s\\ 
&=&(-1)^{\sum\limits_{1\leq i<j\leq l}\lambda_i\lambda_j+\sum\limits_{i=1}^l \lambda_i} \xi_{r_1,[a_1,b_1]}^{\alpha} \cdots \xi_{r_l,[a_l,b_l]}^{\alpha} \sum_{s=1}^{r_1-1} \xi_{s,[k,k]}^{\alpha}T_s\\  
&=&(-1)^{\sum\limits_{1\leq i<j\leq l}\lambda_i\lambda_j+\sum\limits_{i=1}^l\lambda_i}\xi_{r_1,[a_1,b_1]}^{\alpha}\cdots\xi_{r_l,[a_l,b_l]}^{\alpha}(g_1^{\alpha_{k1}}\cdots g_{r_1-1}^{\alpha_{k,r_1-1}}-1),\\ &&\\
&&\sum_{u=1}^{l-1} \sum_{s=r_u+1}^{r_{u+1}-1} (-1)^{\sum\limits_{1\leq i<j\leq l}\lambda_i\lambda_j+\sum\limits_{i=1}^l \lambda_i} \xi_{r_1,[a_1+1,b_1+1]}^{\alpha}\cdots \xi_{r_u,[a_u+1,b_u+1]}^{\alpha} \cdot \notag \\
&& \qquad \xi_{s, [a_u, a_u]} \xi_{r_{u+1},[a_{u+1},b_{u+1}]}^{\alpha} \cdots \xi_{r_l,[a_l,b_l]}^{\alpha}
 (-1)^{\lambda_1 + \cdots + \lambda_u} T_s \notag\\
&=&\sum_{u=1}^{l-1}  (-1)^{\sum\limits_{1\leq i<j\leq l}\lambda_i\lambda_j+\sum\limits_{i=u+1}^l \lambda_i} \xi_{r_1,[a_1+1,b_1+1]}^{\alpha}\cdots \xi_{r_u,[a_u+1,b_u+1]}^{\alpha} \cdot \notag \\
&& \qquad \xi_{r_{u+1},[a_{u+1},b_{u+1}]}^{\alpha} \cdots \xi_{r_l,[a_l,b_l]}^{\alpha}
 \sum_{s=r_u+1}^{r_{u+1}-1}  \xi_{s, [a_u, a_u]} T_s \notag\\%
 &=&\sum_{u=1}^{l-1}  (-1)^{\sum\limits_{1\leq i<j\leq l}\lambda_i\lambda_j+\sum\limits_{i=u+1}^l \lambda_i} \xi_{r_1,[a_1+1,b_1+1]}^{\alpha}\cdots \xi_{r_u,[a_u+1,b_u+1]}^{\alpha} \cdot \notag \\
&& \qquad \xi_{r_{u+1},[a_{u+1},b_{u+1}]}^{\alpha} \cdots \xi_{r_l,[a_l,b_l]}^{\alpha}
 (g_1^{\alpha_{a_u 1}} \cdots g_{r_{u+1}-1} ^{\alpha_{a_u, r_{u+1}-1}} - g_1^{\alpha_{a_u 1}} \cdots g_{r_u}^{\alpha_{a_u, r_u}}), \notag\\%
  &&\\
&&\sum_{s=r_l+1}^n(-1)^{\sum\limits_{1\leq i<j\leq l}\lambda_i\lambda_j+\sum\limits_{i=1}^l \lambda_i} \xi_{r_1,[a_1+1, b_1+1]}^{\alpha} \cdots \xi_{r_l,[a_l+1,b_l+1]}^{\alpha} \xi_{s,[1,1]}^{\alpha}(-1)^{\sum_{i=1}^l\lambda_i}T_s\\
&=&(-1)^{\sum\limits_{1\leq i<j\leq l}\lambda_i\lambda_j} \xi_{r_1,[a_1+1, b_1+1]}^{\alpha} \cdots \xi_{r_l,[a_l+1,b_l+1]}^{\alpha} (g_1^{\alpha_{11}}\cdots g_n^{\alpha_{1n}}-g_1^{\alpha_{11}}\cdots g_{r_l}^{\alpha_{1,r_l}}).
\end{eqnarray*}
With these, \eqref{c2} becomes
\begin{eqnarray} \label{c3}
&&\sum_{1\leq u \leq l}(-1)^{\sum\limits_{1\leq i<j\leq l}\lambda_i\lambda_j+\sum\limits_{j=u+1}^l \lambda_j} \xi_{r_1,[a_1+1,b_1+1]}^{\alpha} \cdots \xi_{r_{u-1},[a_{u-1}+1,b_{u-1}+1]}^{\alpha}  \cdot \\
&&\qquad \xi_{r_u,[a_u,b_u+1]}^{\alpha} \xi_{r_{u+1},[a_{u+1},b_{u+1}]}^{\alpha}  \cdots \xi_{r_l, [a_l, b_l]}^\alpha \E_{r_u^{\lambda_u+1}} \notag \\
&+& (-1)^{\sum\limits_{1\leq i<j\leq l}\lambda_i\lambda_j+\sum\limits_{i=1}^l\lambda_i}\xi_{r_1,[a_1,b_1]}^{\alpha}\cdots\xi_{r_l,[a_l,b_l]}^{\alpha}(g_1^{\alpha_{k1}}\cdots g_{r_1-1}^{\alpha_{k,r_1-1}}-1) \notag \\
&+& \sum_{u=1}^{l-1}  (-1)^{\sum\limits_{1\leq i<j\leq l}\lambda_i\lambda_j+\sum\limits_{i=u+1}^l \lambda_i} \xi_{r_1,[a_1+1,b_1+1]}^{\alpha}\cdots \xi_{r_u,[a_u+1,b_u+1]}^{\alpha} \cdot \notag \\
&& \qquad \xi_{r_{u+1},[a_{u+1},b_{u+1}]}^{\alpha} \cdots \xi_{r_l,[a_l,b_l]}^{\alpha}
 (g_1^{\alpha_{a_u 1}} \cdots g_{r_{u+1}-1} ^{\alpha_{a_u, r_{u+1}-1}} - g_1^{\alpha_{a_u 1}} \cdots g_{r_u}^{\alpha_{a_u, r_u}}) \notag \\
&+& (-1)^{\sum\limits_{1\leq i<j\leq l}\lambda_i\lambda_j} \xi_{r_1,[a_1+1, b_1+1]}^{\alpha} \cdots \xi_{r_l,[a_l+1,b_l+1]}^{\alpha} (g_1^{\alpha_{11}}\cdots g_n^{\alpha_{1n}}-g_1^{\alpha_{11}}\cdots g_{r_l}^{\alpha_{1,r_l}}). \notag 
\end{eqnarray}
 
We need to prove that the two formulas \eqref{c1'} and \eqref{c3} are equal. 
By cancelling their obvious common terms, namely the first two terms of \eqref{c1'}, it suffices to prove
\begin{eqnarray*}
&&\sum_{i=1}^l  \xi_{r_1,[a_1+1,b_1+1]}^{\alpha}\cdots\xi_{r_{i-1},[a_{i-1}+1,b_{i-1}+1]}^{\alpha} \xi_{r_{i+1},[a_{i+1},b_{i+1}]}^{\alpha}\cdots\xi_{r_l,[a_l,b_l]}^{\alpha} \sum_{u=a_i}^{b_i}(-1)^u  \xi_{r_i,[a_i,b_i]}^{\alpha'_u} \\
&=&\sum_{u=1}^l (-1)^{\sum\limits_{i=u+1}^l\lambda_i} \xi_{r_1,[a_1+1,b_1+1]}^{\alpha} \cdots \xi_{r_{u-1},[a_{u-1}+1,b_{u-1}+1]}^{\alpha} \xi_{r_u,[a_u,b_u+1]}^{\alpha} \xi_{r_{u+1},[a_{u+1},b_{u+1}]}^{\alpha}  \cdots \xi_{r_l, [a_l, b_l]}^\alpha \E_{r_u^{\lambda_u+1}} \notag \\
&&+ \sum_{u=1}^{l-1}  (-1)^{\sum\limits_{i=u+1}^l \lambda_i} \xi_{r_1,[a_1+1,b_1+1]}^{\alpha}\cdots \xi_{r_u,[a_u+1,b_u+1]}^{\alpha}  \xi_{r_{u+1},[a_{u+1},b_{u+1}]}^{\alpha} \cdots \xi_{r_l,[a_l,b_l]}^{\alpha} \cdot \notag \\
&& \qquad (g_1^{\alpha_{a_u 1}} \cdots g_{r_{u+1}-1} ^{\alpha_{a_u, r_{u+1}-1}} - g_1^{\alpha_{a_u 1}} \cdots g_{r_u}^{\alpha_{a_u, r_u}}) \notag \\
&&+ (-1)^{\sum\limits_{i=1}^l\lambda_i}\xi_{r_1,[a_1,b_1]}^{\alpha}\cdots\xi_{r_l,[a_l,b_l]}^{\alpha} g_1^{\alpha_{k1}}\cdots g_{r_1-1}^{\alpha_{k,r_1-1}} - \xi_{r_1,[a_1+1, b_1+1]}^{\alpha} \cdots \xi_{r_l,[a_l+1,b_l+1]}^{\alpha} g_1^{\alpha_{11}}\cdots g_{r_l}^{\alpha_{1,r_l}}. \notag 
\end{eqnarray*}
Note  that the latter is equal to
\begin{eqnarray*}
&&\sum_{u=1}^l (-1)^{\sum\limits_{i=u+1}^l \lambda_i} \xi_{r_1,[a_1+1,b_1+1]}^{\alpha} \cdots \xi_{r_{u-1},[a_{u-1}+1,b_{u-1}+1]}^{\alpha} \xi_{r_u,[a_u,b_u+1]}^{\alpha} \xi_{r_{u+1},[a_{u+1},b_{u+1}]}^{\alpha}  \cdots \xi_{r_l, [a_l, b_l]}^\alpha  \E_{r_u^{\lambda_u+1}}\\
&&+ \sum_{u=1}^{l}  (-1)^{\sum\limits_{i=u}^l \lambda_i} \xi_{r_1,[a_1+1,b_1+1]}^{\alpha}\cdots \xi_{r_{u-1},[a_{u-1}+1,b_{u-1}+1]}^{\alpha}  \xi_{r_{u},[a_{u},b_{u}]}^{\alpha} \cdots \xi_{r_l,[a_l,b_l]}^{\alpha} g_1^{\alpha_{a_{u-1} 1}} \cdots g_{r_{u}-1} ^{\alpha_{a_{u-1}, r_{u}-1}} \notag \\
&&- \sum_{u=1}^{l}  (-1)^{\sum\limits_{i=u+1}^l \lambda_i} \xi_{r_1,[a_1+1,b_1+1]}^{\alpha}\cdots \xi_{r_u,[a_u+1,b_u+1]}^{\alpha}  \xi_{r_{u+1},[a_{u+1},b_{u+1}]}^{\alpha} \cdots \xi_{r_l,[a_l,b_l]}^{\alpha} g_1^{\alpha_{a_u 1}} \cdots g_{r_u}^{\alpha_{a_u, r_u}} \notag \\
&=&\sum_{u=1}^l  \xi_{r_1,[a_1+1,b_1+1]}^{\alpha}\cdots\xi_{r_{u-1},[a_{u-1}+1,b_{u-1}+1]}^{\alpha} \xi_{r_{u+1},[a_{u+1},b_{u+1}]}^{\alpha}\cdots\xi_{r_l,[a_l,b_l]}^{\alpha} \cdot \left( (-1)^{\sum\limits_{i=u+1}^l \lambda_i} \xi_{r_u,[a_u,b_u+1]}^{\alpha} \E_{r_u^{\lambda_u+1}}
\right. \\
&&\left. \qquad + (-1)^{\sum\limits_{i=u}^l \lambda_i}  \xi_{r_{u},[a_{u},b_{u}]}^{\alpha}  g_1^{\alpha_{a_{u-1} 1}} \cdots g_{r_{u}-1} ^{\alpha_{a_{u-1}, r_{u}-1}} 
- (-1)^{\sum\limits_{i=u+1}^l \lambda_i}  \xi_{r_u,[a_u+1,b_u+1]}^{\alpha}  g_1^{\alpha_{a_u 1}} \cdots g_{r_u}^{\alpha_{a_u, r_u}} \right).
\end{eqnarray*}
Now it is enough to verify that 
\begin{eqnarray} \label{me}
 &&\sum_{u=a_i}^{b_i}(-1)^u  \xi_{r_i,[a_i,b_i]}^{\alpha'_u}=(-1)^{\sum\limits_{i=u+1}^l \lambda_i} \xi_{r_u,[a_u,b_u+1]}^{\alpha} \E_{r_u^{\lambda_u+1}} + (-1)^{\sum\limits_{i=u}^l \lambda_i}  \xi_{r_{u},[a_{u},b_{u}]}^{\alpha}  g_1^{\alpha_{a_{u-1} 1}} \cdots g_{r_{u}-1} ^{\alpha_{a_{u-1}, r_{u}-1}}  \\
&&\qquad \qquad \qquad \qquad \qquad - (-1)^{\sum\limits_{i=u+1}^l \lambda_i}  \xi_{r_u,[a_u+1,b_u+1]}^{\alpha}  g_1^{\alpha_{a_u 1}} \cdots g_{r_u}^{\alpha_{a_u, r_u}}. \notag
\end{eqnarray}
The verification is split into two cases. If $b_i-a_i$ is even, then the equality is immediate simply by noting that
\[a_u = {\sum_{i=u+1}^l \lambda_i}+1, \quad b_u = {\sum_{i=u}^l \lambda_i} . \] If $b_i-a_i$ is odd, noting that
\begin{eqnarray*} 
(-1)^{\sum\limits_{i=u+1}^l \lambda_i} \xi_{r_u,[a_u,b_u+1]}^{\alpha} \E_{r_u^{\lambda_u+1}} &=& (-1)^{\sum\limits_{i=u+1}^l \lambda_i} \xi_{r_u,[a_u+1,b_u+1]}^{\alpha} g_1^{\alpha_{a_u,1}} \cdots g_{r_u-1}^{\alpha_{a_u,r_u-1}}(g_{r_u})_{\alpha_{a_u,r_u}} (g_{r_u}-1)  \\
&=&  (-1)^{\sum\limits_{i=u+1}^l \lambda_i} \xi_{r_u,[a_u+1,b_u+1]}^{\alpha} g_1^{\alpha_{a_u,1}} \cdots g_{r_u-1}^{\alpha_{a_u,r_u-1}} (g_{r_u}^{\alpha_{a_u,r_u}} -1), 
\end{eqnarray*} then the equality \eqref{me} follows.

The proof is completed.
\end{proof}

\subsection{Normalized cocycles}
Denote
$$\eta_{r,[a,b]}^{\alpha}:=\left\{
\begin{array}{ll}[\frac{\alpha_{br}+\alpha_{b-1,r}}{m_{r}}]\cdots[\frac{\alpha_{a+1,r}+\alpha_{ar}}{m_{r}}],&\;\;\;b-a\text{ odd;}\\

[\frac{\alpha_{br}+\alpha_{b-1,r}}{m_{r}}]\cdots[\frac{\alpha_{a+2,r}+\alpha_{a+1,r}}{m_{r}}]\alpha_{ar},&\;\;\;b-a\text{ even.}
\end{array}\right.$$

\begin{corollary}
The following $k$-cochains $\o\in \Hom_{\mathbb{Z}G}(B_k,\k^*)$ given by
\begin{equation}\label{barcocycleabelian}
\o([g^{\alpha_1}, \dots, g^{\alpha_{k}}])=\prod_{l=1}^{k}\prod_{\begin{array}{ccc}1\leq r_{1}<\cdots<r_{l}\leq n\\\lambda_1+\cdots+\lambda_l=k,\lambda_1\emph{ odd}\\\lambda_i\ge1\emph{ for }1\le i\le l\end{array}}\zeta_{m_{r_1}}^{(-1)^{\sum\limits_{1\leq i<j\leq l}\lambda_{i}\lambda_{j}}\eta_{r_{1},[a_{1},b_{1}]}^{\alpha}\cdots\eta_{r_{l},[a_{l},b_{l}]}^{\alpha}a_{r_1^{\lambda_1}\cdots r_l^{\lambda_l}}}
\end{equation}
where $a_u = {\sum\limits_{i=u+1}^l \lambda_i}+1, \quad b_u = {\sum\limits_{i=u}^l \lambda_i}$ and $0\leq a_{r_1^{\lambda_1}\cdots r_l^{\lambda_l}}<m_{r_1}$ for $1\leq r_1<\cdots<r_l\leq n$
form a complete set of representatives of $k$-cocycles of the complex $(B_{\bullet}^*,\partial_{\bullet}^*)$.
\end{corollary}
\begin{proof}
It follows from the chain map \eqref{Fkabelian} and Theorem \ref{complete}. 
\end{proof}

\subsection{A chain map  from $(K_{\bullet},d_{\bullet})$ to $(B_{\bullet},\partial_{\bullet})$}
For completeness, we also include a chain map from the Koszul-like resolution $(K_{\bullet},d_{\bullet})$ to the normalized bar resolution $(B_{\bullet},\partial_{\bullet}).$ This chain map is very useful for comparing cohomology classes of normalized cocycles and for studying the whole cohomology ring structure, etc. 

Denote an ordered set of $\lambda$ elements as 
$$\Lambda_{r^{\lambda}}:=\left\{\begin{array}{ll}(N_r,g_r,N_r,g_r,\dots,N_r,g_r),&\lambda\text{ even;}\\(g_r,N_r,g_r,N_r,g_r\dots,N_r,g_r),&\lambda\text{ odd.}\end{array}\right.$$
Given a set of positive integers $\lambda_1, \ \lambda_2, \dots, \lambda_l$ with $\lambda_1+\cdots+\lambda_l=k$,
let $\Shuffle(\lambda_1,\dots,\lambda_l)$ be the subset of the permutation group $S_k$ such that the elements of it preserve the order of elements of each block of the partition $(\lambda_1,\dots,\lambda_l)$. For each $k,$ define a map 
\begin{eqnarray*}
 G_{k} \colon K_{k} &\To& B_{k}  \\
\Phi_{r_1^{\lambda_1}\cdots r_l^{\lambda_l}} &\mapsto& \sum_{\sigma\in \Shuffle(\lambda_1,\dots,\lambda_l)}(-1)^{\sigma}[\sigma(\Lambda_{r_1^{\lambda_1}},\dots,\Lambda_{r_l^{\lambda_l}})].
\end{eqnarray*}

\begin{proposition} 
We have the following commutative diagram
\begin{figure}[hbt]
\begin{picture}(150,50)(50,-40)
\put(0,0){\makebox(0,0){$ \cdots$}}\put(10,0){\vector(1,0){20}}\put(40,0){\makebox(0,0){$K_{3}$}}
\put(50,0){\vector(1,0){20}}\put(80,0){\makebox(0,0){$K_{2}$}}
\put(90,0){\vector(1,0){20}}\put(120,0){\makebox(0,0){$K_{1}$}}
\put(130,0){\vector(1,0){20}}\put(160,0){\makebox(0,0){$K_{0}$}}
\put(170,0){\vector(1,0){20}}\put(200,0){\makebox(0,0){$\mathbb{Z}$}}
\put(210,0){\vector(1,0){20}}\put(240,0){\makebox(0,0){$0$}}

\put(0,-40){\makebox(0,0){$ \cdots$}}\put(10,-40){\vector(1,0){20}}\put(40,-40){\makebox(0,0){$B_{3}$}}
\put(50,-40){\vector(1,0){20}}\put(80,-40){\makebox(0,0){$B_{2}$}}
\put(90,-40){\vector(1,0){20}}\put(120,-40){\makebox(0,0){$B_{1}$}}
\put(130,-40){\vector(1,0){20}}\put(160,-40){\makebox(0,0){$B_{0}$}}
\put(170,-40){\vector(1,0){20}}\put(200,-40){\makebox(0,0){$\mathbb{Z}$}}
\put(210,-40){\vector(1,0){20}}\put(240,-40){\makebox(0,0){$0$}}

\put(40,-10){\vector(0,-1){20}}
\put(80,-10){\vector(0,-1){20}}
\put(120,-10){\vector(0,-1){20}}
\put(158,-10){\line(0,-1){20}}\put(160,-10){\line(0,-1){20}}
\put(198,-10){\line(0,-1){20}}\put(200,-10){\line(0,-1){20}}

\put(60,5){\makebox(0,0){$d$}}
\put(100,5){\makebox(0,0){$d$}}
\put(140,5){\makebox(0,0){$d$}}

\put(60,-35){\makebox(0,0){$\partial_3$}}
\put(100,-35){\makebox(0,0){$\partial_2$}}
\put(140,-35){\makebox(0,0){$\partial_1$}}

\put(50,-20){\makebox(0,0){$G_{3}$}}
\put(90,-20){\makebox(0,0){$G_{2}$}}
\put(130,-20){\makebox(0,0){$G_{1}$}}

\end{picture}               
\end{figure}
\end{proposition}

\begin{proof}
By direct verification similarly as the proof of Proposition \ref{chainmap}. The detail is omitted.
\end{proof}

\subsection{A translation to quantum field theory} \label{subspt}
Now we follow the notations in \cite{spt} and translate our result into quantum field theory language. Let $G=\mathbb{Z}_{N_1}\times\cdots\times\mathbb{Z}_{N_n}$ where $N_i|N_{i+1}$ for $1\le i\le n-1$.
Let $k=d+1$ be the spacetime dimension. For $1\le l\le d+1$, $1\le r_1<\cdots<r_l\le n$, $\lambda_i\ge1$ for $1\le i\le l$,
define $$\phi_{{r_i}^{\lambda_i}}=\left\{\begin{array}{ll}A_{r_i}dA_{r_i}\cdots dA_{r_i}, \quad &\text{ if }\lambda_i\text{ odd;}\\ 
dA_{r_i}\cdots dA_{r_i}, \quad &\text{ if }\lambda_i\text{ even.}\end{array}\right.$$
We generalize the correspondence between the partition functions of fields and cocycles given in \cite{spt}.

The generalized correspondence connects the part $$\zeta_{N_{r_1}}^{(-1)^{\sum\limits_{1\leq i<j\leq l}\lambda_{i}\lambda_{j}}\eta_{r_{1},[a_{1},b_{1}]}^{\alpha}\cdots\eta_{r_{l},[a_{l},b_{l}]}^{\alpha}a_{r_1^{\lambda_1}\cdots r_l^{\lambda_l}}}$$ 
of $(d+1)$-cocycle $\o_{d+1}$
and the partition function $$\zeta_{N_{r_1}}^{(-1)^{\sum\limits_{1\leq i<j\leq l}\lambda_{i}\lambda_{j}}a_{r_1^{\lambda_1}\cdots r_l^{\lambda_l}}\frac{N_{r_1}^{\lambda_1-2[\frac{\lambda_1}{2}]}\cdots N_{r_l}^{\lambda_l-2[\frac{\lambda_l}{2}]}}{(2\pi)^{[\frac{\lambda_1+1}{2}]+\cdots+[\frac{\lambda_l+1}{2}]}}\int\phi_{r_l^{\lambda_l}}\cdots\phi_{r_1^{\lambda_1}}}$$ 
where the corresponding terms of $A_u$ and $dA_u$ are given in \cite{spt} and the order of $A_u$ and $dA_u$ is so arranged that their positions indicate which component of $\alpha$ they correspond to. Note that $\alpha=(\alpha_1,\dots,\alpha_{d+1})$, $\lambda_1+\cdots+\lambda_l=d+1$, and $\lambda_1$  is odd. Our result reveals the fact that we don't need higher form fields $B,C$, etc, to get a complete set of representatives of cocycles.

Now we explain how we get these partition functions.
First, any 1-form field is the linear combination of the wedge products of some $A_u$ and $dA_u$ where each $A_u$ appears at most once, i.e. the linear combination of $\phi_{r_l^{\lambda_l}}\cdots\phi_{r_1^{\lambda_1}}$ for some $1\le l\le d+1$, $1\le r_1<\cdots<r_l\le n$, $\lambda_i\ge1$ for $1\le i\le l$.
After integration by part on $\int\phi_{r_l^{\lambda_l}}\cdots\phi_{r_1^{\lambda_1}}$, we need only consider those terms with $\lambda_1$ odd.

Due to a discrete $\mathbb{Z}_N$ gauge symmetry, and the gauge transformation must be identified by $2\pi$, we have the following general rules:
$$\oint A_u=\frac{2\pi n_u}{N_u}\mod 2\pi, \quad \quad \quad \oint\delta A_u=0\mod 2\pi.$$

We consider a spacetime with a volume size $L^{d+1}$ where $L$ is the length of one dimension, for example $T^{d+1}$ torus.
The allowed large gauge transformation implies that locally $A$ can be:
$$A_{u,\mu}=\frac{2\pi n_udx_{\mu}}{N_uL}, \quad \,\,\,\delta A_u=\frac{2\pi m_udx_{\mu}}{L}.$$
Now we consider the partition function $\exp(ik_{r_1^{\lambda_1}\cdots r_l^{\lambda_l}}\int\phi_{r_l^{\lambda_l}}\cdots\phi_{r_1^{\lambda_1}})$ with $\lambda_1$ odd.
Note that $\delta(dA_u)=0.$ Thus for the large gauge transformation, we have $k_{r_1^{\lambda_1}\cdots r_l^{\lambda_l}}\int\delta(\phi_{r_l^{\lambda_l}}\cdots\phi_{r_1^{\lambda_1}})=0\mod 2\pi$.
This implies $$k_{r_1^{\lambda_1}\cdots r_l^{\lambda_l}}=p_{r_1^{\lambda_1}\cdots r_l^{\lambda_l}}\frac{N_{r_2}^{\lambda_2-2[\frac{\lambda_2}{2}]}\cdots N_{r_l}^{\lambda_l-2[\frac{\lambda_l}{2}]}}{(2\pi)^{[\frac{\lambda_1+1}{2}]+\cdots+[\frac{\lambda_l+1}{2}]-1}}$$
where $p_{r_1^{\lambda_1}\cdots r_l^{\lambda_l}}\in\mathbb{Z}$.

For the flux identification, we have $$k_{r_1^{\lambda_1}\cdots r_l^{\lambda_l}}\int\phi_{r_l^{\lambda_l}}\cdots\phi_{r_1^{\lambda_1}}=\frac{(2\pi)^{[\frac{\lambda_1+1}{2}]+\cdots+[\frac{\lambda_l+1}{2}]}n_{r_1}^{[\frac{\lambda_1+1}{2}]}\cdots n_{r_l}^{[\frac{\lambda_l+1}{2}]}}{N_{r_1}^{\lambda_1-2[\frac{\lambda_1}{2}]}\cdots N_{r_l}^{\lambda_l-2[\frac{\lambda_l}{2}]}}.$$
Hence $$(2\pi)^{[\frac{\lambda_1+1}{2}]+\cdots+[\frac{\lambda_l+1}{2}]-1}k_{r_1^{\lambda_1}\cdots r_l^{\lambda_l}}\simeq(2\pi)^{[\frac{\lambda_1+1}{2}]+\cdots+[\frac{\lambda_l+1}{2}]-1}k_{r_1^{\lambda_1}\cdots r_l^{\lambda_l}}+N_{r_1}^{\lambda_1-2[\frac{\lambda_1}{2}]}\cdots N_{r_l}^{\lambda_l-2[\frac{\lambda_l}{2}]}.$$
Here $\simeq$ means the level identification.
Therefore, the cyclic period of $p_{r_1^{\lambda_1}\cdots r_l^{\lambda_l}}$ is $N_{r_1}$.

Finally let $(-1)^{\sum\limits_{1\leq i<j\leq l}\lambda_{i}\lambda_{j}}p_{r_1^{\lambda_1}\cdots r_l^{\lambda_l}}=a_{r_1^{\lambda_1}\cdots r_l^{\lambda_l}}$.
Then we get the partition functions in correspondence with cocycles.

\section{On Braided linear Gr-categories}
The monoidal category of finite dimensional vector spaces graded by a group $G,$ with the usual tensor product and associativity constraint given by a 3-cocycle $\o$ is denoted by $\Vec^\o_G.$ Such a monoidal category is called a linear Gr-category. The terminology goes back to  Ho\`{a}ng Xu\^{a}n S\'{i}nh \cite{grc}, a student of Grothendieck. The aim of this section is give a complete description to all braided linear Gr-categories with a help of the explicit formulas of normalized $3$-cocycles. This extends the related partial results obtained in \cite{bct, hly, js, ms} to full generality. 

\subsection{Monoidal structures}
Recall that the category $\Vec_G$ of finite-dimensional $G$-graded vector spaces has simple objects $\{S_g|g \in G\}$ where $(S_g)_h=\delta_{g,h}\k, \ \forall h \in G.$ The tensor product is given by $S_g \otimes S_h=S_{gh},$ and $S_1$ ($1$ is the identity of $G$) is the unit object. Without loss of generality we may assume that the left and right unit constraints are identities. If $a$ is an associativity constraint on $\Vec_G,$ then it is given by $a_{S_f,S_g,S_h}=\o(f,g,h)\id,$ where $\o:G \times G \times G \rightarrow \k^*$ is a function. The pentagon axiom and the triangle axiom give
\begin{gather*}
\o(ef,g,h)\o(e,f,gh)=\o(e,f,g)\o(e,fg,h)\o(f,g,h), \\ \o(f,1,g)=1,
\end{gather*}
which say exactly that $\o$ is a normalized 3-cocycle on $G.$ Note that cohomologous cocycles define equivalent monoidal structures, therefore the classification of monoidal structures on $\Vec_G$ is equivalent to determining a complete set of representatives of normalized 3-cocycles on $G.$

\subsection{Normalized 3-cocycles} 
In the special case $k=3$, if we abbreviate $a_{r^3}$ by $a_r$, $a_{rs^2}$ by $a_{rs}$, then \eqref{barcocycleabelian} becomes
\begin{eqnarray}\label{3cocycle}
&& \o:\;B_{3}\To \k^{\ast}  \\ \notag
&&[g_{1}^{i_{1}}\cdots g_{n}^{i_{n}},g_{1}^{j_{1}}\cdots g_{n}^{j_{n}},g_{1}^{k_{1}}\cdots g_{n}^{k_{n}}] \mapsto \prod_{r=1}^{n}\zeta_{m_r}^{a_{r}i_{r}[\frac{j_{r}+k_{r}}{m_{r}}]}
\prod_{1\leq r<s\leq n}\zeta_{m_r}^{a_{rs}k_{r}[\frac{i_{s}+j_{s}}{m_{s}}]}
\prod_{1\leq r<s<t\leq n}\zeta_{m_r}^{-a_{rst}k_{r}j_{s}i_{t}} 
\end{eqnarray}
where $0\le a_r,a_{rs},a_{rst}<m_r$.

\begin{remark}
The present complete set of representatives of normalized $3$-cocycles is slightly different from that in \cite{bgrc1, hly}. Of course they are equivalent up to cohomology.
\end{remark}

\subsection{Braided structures} Now we consider the braided structures on a linear Gr-category $\Vec_G^\o.$ Recall that a braiding in $\Vec_G^\o$ is a commutativity constraint $c: \otimes \rightarrow \otimes^{\op}$ satisfying the hexagon axiom. Note that $c$ is given by $c_{S_x,S_y}=\R(x,y)\id,$ where $\R: G \times G \rightarrow \k^*$ is a function, and the hexagon axiom of $c$ says that
\begin{equation}\label{EM3cocycle}
\frac{\R(xy,z)}{\R(x,z)\R(y,z)}\frac{\o(x,z,y)}{\o(x,y,z)\o(z,x,y)}=1=\frac{\R(x,yz)}{\R(x,y)\R(x,z)}\frac{\o(x,y,z)\o(y,z,x)}{\o(y,x,z)}
\end{equation}
for all $x,y,z\in G$.

In other words, $\R$ is a quasi-bicharacter of $G$ with respect to $\o.$ Therefore, the classification of braidings in $\Vec_G^\o$ is equivalent to determining all the quasi-bicharacters of $G$ with respect to $\o.$ It is interesting to remark that the braided monoidal structures $(\o,\R)$ on $\Vec_G$ appeared already in the 1950s in terms of the so-called abelian cohomology of Eilenberg and Mac Lane \cite{EMa, EMb}.

\subsection{Quasi-bicharacters}  Clearly, any quasi-bicharacter $\R$ is uniquely determined by the following values:
$$r_{ij}:=\R(g_{i},g_{j}),\;\;\;\;\;\;\;\;\textrm{for all}\;\;1\leq i, \ j\leq n.$$

\begin{proposition} \label{condition}
Let $r_{ij}\in \k^{\ast}$ for
$1\leq i, \ j\leq n$. Then there is a quasi-bicharacter $\R$ with respect to $\o$ satisfying $\R(g_{i},g_{j})=r_{ij}$
if and only if the following equations are satisfied:
\begin{eqnarray*}
r_{ii}^{m_{i}}=\zeta_{m_{i}}^{a_{i}}=\zeta_{m_{i}}^{-a_{i}},&&\emph{for}\;\;1\leq i\leq n,\\
r_{ij}^{m_{i}}=r_{ji}^{m_i}=1,\ \ a_{ij}=0,&&\emph{for}\;\;1\le i<j\le n,\\
a_{rst}=0,&&\emph{for}\;\;1\leq r<s<t\leq n.
\end{eqnarray*}
\end{proposition}

\begin{proof}  ``$\Rightarrow$".   
For the case $r<s<t,$ consider $\R(g_{t}g_{s},g_{r})$ and $\R(g_{s}g_{t},g_{r})$ which obviously are equal. By \eqref{EM3cocycle}, we have
\begin{eqnarray*}\R(g_{t}g_{s}, g_{r})&=&\R(g_{t} , g_{r})\R(g_{s}, g_{r})\frac{\o(g_{r},g_{t},g_{s})\o(g_{t},g_{s},g_{r})}{\o(g_{t},g_{r},g_{s})}\\
&=&\R(g_{t} , g_{r})\R(g_{s}, g_{r})\zeta_{m_r}^{-a_{rst}},\end{eqnarray*}
\begin{eqnarray*}\R(g_{s}g_{t}, g_{r})&=&\R(g_{s} , g_{r})\R(g_{t}, g_{r})\frac{\o(g_{r},g_{s},g_{t})\o(g_{s},g_{t},g_{r})}{\o(g_{s},g_{r},g_{t})}\\
&=&\R(g_{s} , g_{r})\R(g_{t}, g_{r}).\end{eqnarray*}
Therefore, $\zeta_{m_r}^{-a_{rst}}=1$. Since $0\le a_{rst}<m_r$, we arrive at $a_{rst}=0.$

For any $1\leq i\leq n$, applying \eqref{EM3cocycle} iteratively, we have $\R(g_{i},g_{i}^{s})=\R(g_{i},g_{i})^{s}$ and $\R(g_{i}^{s},g_{i})=\R(g_{i},g_{i})^{s}$
for $1\leq s\leq m_{i}-1$. Then
$$1=\R(g_{i},g_{i}^{m_{i}})=\R(g_{i},g_{i})\R(g_{i},g_{i}^{m_{i}-1})
\frac{\o(g_{i},g_{i},g_{i}^{m_{i}-1})}{\o(g_i,g_i,g_i^{m_i-1})\o(g_i,g_i^{m_i-1},g_i)}=\R(g_{i},g_{i})^{m_{i}}\zeta_{m_{i}}^{-a_{i}},$$
$$1=\R(g_{i}^{m_{i}},g_{i})=\R(g_{i}^{m_{i}-1},g_{i})\R(g_{i},g_{i})\frac{\o(g_i^{m_i-1},g_i,g_i)\o(g_{i},g_{i}^{m_{i}-1},g_{i})}{\o(g_i^{m_i-1},g_i,g_i)}=\R(g_{i},g_{i})^{m_{i}}{\zeta_{m_{i}}^{a_{i}}}.$$
Thus $r_{ii}^{m_{i}}=\zeta_{m_{i}}^{a_{i}}=\zeta_{m_{i}}^{-a_{i}}$.

Assume $i<j.$ Applying \eqref{EM3cocycle} iteratively, one has $\R(g_{i}^{k},g_{j})=\R(g_{i},g_{j})^{k}$ for $1\leq k\leq m_{i}-1.$ Therefore,
$$1=\R(g_{i}^{m_{i}},g_{j})=\R(g_{i}^{m_{i}-1},g_{j})\R(g_{i},g_{j})\frac{\o(g_i^{m_i-1},g_i,g_j)\o(g_j,g_i^{m_i-1},g_i)}{\o(g_i^{m_i-1},g_j,g_i)}
=\R(g_{i},g_{j})^{m_{i}}.$$
This implies that $r_{ij}^{m_{i}}=1$. Applying \eqref{EM3cocycle} iteratively, one has $\R(g_{i},g_{j}^{k})=\R(g_{i},g_{j})^{k}$ for $1\leq k\leq m_{j}-1$. Therefore,
$$1=\R(g_{i},g_{j}^{m_{j}})=\R(g_{i},g_{j})\R(g_i,g_j^{m_j-1})\frac{\o(g_j,g_i,g_j^{m_j-1})}{\o(g_i,g_j,g_j^{m_j-1})\o(g_j,g_j^{m_j-1},g_i)}=r_{ij}^{m_{j}}\zeta_{m_i}^{-a_{ij}}.$$ This implies that $r_{ij}^{m_{j}}=\zeta_{m_{i}}^{a_{ij}}$. 
Since $m_i|m_j$, we have $\zeta_{m_i}^{a_{ij}}=1$. Since $0\le a_{ij}<m_i$, we arrive at $a_{ij}=0$.

Assume $i>j.$ Applying \eqref{EM3cocycle} iteratively, one has $\R(g_{i}^{k},g_{j})=\R(g_{i},g_{j})^{k}$ for $1\leq k\leq m_{i}-1.$ Therefore,
$$1=\R(g_{i}^{m_{i}},g_{j})=\R(g_{i}^{m_{i}-1},g_{j})\R(g_{i},g_{j})\frac{\o(g_i^{m_i-1},g_i,g_j)\o(g_j,g_i^{m_i-1},g_i)}{\o(g_i^{m_i-1},g_j,g_i)}
=\R(g_{i},g_{j})^{m_{i}}\zeta_{m_j}^{a_{ij}}.$$
This implies that $r_{ij}^{m_{i}}=\zeta_{m_j}^{-a_{ij}}=1$. Applying \eqref{EM3cocycle} iteratively, one has $\R(g_{i},g_{j}^{k})=\R(g_{i},g_{j})^{k}$ for $1\leq k\leq m_{j}-1$. Therefore,
$$1=\R(g_{i},g_{j}^{m_{j}})=\R(g_{i},g_{j})\R(g_i,g_j^{m_j-1})\frac{\o(g_j,g_i,g_j^{m_j-1})}{\o(g_i,g_j,g_j^{m_j-1})\o(g_j,g_j^{m_j-1},g_i)}=r_{ij}^{m_{j}}.$$ This implies that $r_{ij}^{m_{j}}=1$.

The necessity is proved.

``$\Leftarrow$". Conversely, define a  map $\R: G\times G \To k^{\ast}$  by setting $$\R(g_{1}^{i_{1}}\cdots g_{n}^{i_{n}},g_{1}^{j_{1}}\cdots g_{n}^{j_{n}}):=\prod_{s=1}^n r_{ss}^{i_{s}j_{s}}.$$ We verify that $\R$ is a quasi-bicharacter of $G$ with respect to $\o.$

Let $x=g_{1}^{i_{1}}\cdots g_{n}^{i_{n}}, \ y=g_{1}^{j_{1}}\cdots g_{n}^{j_{n}}, \ z=g_{1}^{k_{1}}\cdots g_{n}^{k_{n}},$ then
$$\R(g_{1}^{i_{1}}\cdots g_{n}^{i_{n}}\cdot g_{1}^{j_{1}}\cdots g_{n}^{j_{n}},g_{1}^{k_{1}}\cdots g_{n}^{k_{n}})=\prod_{s=1}^n r_{ss}^{(i_{s}+j_{s})'k_{s}},$$ where $(i_{s}+j_{s})'$ denotes the remainder of division of $i_{s}+j_{s}$ by $m_{s}.$
Consider $\R(x , z)\R(y ,z)\frac{\o(z,x,y)\o(x,y,z)}{\o(x,z,y)}.$ By direct calculation, one has
$$\frac{\o(z,x,y)\o(x,y,z)}{\o(x,z,y)}=\prod_{l=1}^{n}\zeta_{m_{l}}^{a_{l}k_{l}
[\frac{i_{l}+j_{l}}{m_l}]}.$$
Therefore,
\begin{eqnarray*}
&&\R(x,z)\R(y,z)\frac{\o(z,x,y)\o(x,y,z)}{\o(x,z,y)}\\
&=&\prod_{s=1}^n r_{ss}^{(i_{s}+j_{s})k_{s}}\prod_{l=1}^{n}\zeta_{m_{l}}^{a_{l}k_{l}
[\frac{i_{l}+j_{l}}{m_{l}}]}\\
&=&\prod_{s=1}^n r_{ss}^{((i_{s}+j_{s})'+[\frac{i_{s}+j_{s}}{m_{s}}]m_{s})k_{s}}\prod_{l=1}^{n}\zeta_{m_{l}}^{a_{l}k_{l}
[\frac{i_{l}+j_{l}}{m_{l}}]}\\
&=&\prod_{s=1}^n r_{ss}^{(i_{s}+j_{s})'k_{s}}\prod_{l=1}^{n} c_{ll}^{[\frac{i_{l}+j_{l}}{m_{l}}]m_{l}k_{l}}\prod_{l=1}^{n}\zeta_{m_{l}}^{a_{l}k_{l}
[\frac{i_{l}+j_{l}}{m_{l}}]}\\
&=&\prod_{s=1}^n r_{ss}^{(i_{s}+j_{s})'k_{s}}\prod_{l=1}^{n}\zeta_{m_{l}}^{-a_{l}k_{l}
[\frac{i_{l}+j_{l}}{m_{l}}]}\prod_{l=1}^{n}\zeta_{m_{l}}^{a_{l}k_{l}
[\frac{i_{l}+j_{l}}{m_{l}}]}.
\end{eqnarray*}
This implies that $$\R(x,z)\R(y,z)\frac{\o(z,x,y)\o(x,y,z)}{\o(x,z,y)}=\prod_{s=1}^n r_{ss}^{(i_{s}+j_{s})'k_{s}}=\R(xy,z).$$

The sufficiency is proved.
\end{proof}

\section{The Dijkgraaf-Witten invariant of the $n$-torus}
In this section, we give a formula of the Dijkgraaf-Witten invariant for an arbitrary $n$-torus $T^n$ associated to finite abelian groups. In the special case of $n=2,$ we recover and improve some results obtained in \cite{turaev}. This is due to the fact that as we have an explicit formula of $2$-cocycles, we are able to derive dimension formulas for irreducible projective representations of finite abelian groups. This is of independent interest.

\subsection{Definition of DW invariants}
Just as its name implies, such an invariant of 3-manifolds was introduced by Dijkgraaf and Witten in \cite{dw}. Then it was generalized to arbitrary dimension in \cite{freed} by Freed.

Now we recall briefly the definition of DW invariants. The reader is referred to \cite{dw, freed, turaev} for more details.
Let $G$ be a finite group and let $[\o]\in \H^n(BG;\k^*)$. For a closed oriented $n$-manifold $M$, the Dijkgraaf-Witten invariant of $M$ is defined as 
$$Z^{[\o]}(M)=\frac{1}{|G|}\sum_{\phi:\pi_1(M)\to G}\langle f_{\phi}^*[\o],[M]\rangle,$$
where $f_{\phi}:M\to BG$ is a map inducing $\phi$ on the fundamental group which is determined by $\phi$ up to homotopy, $[M]$ is the fundamental class of $M$, and $\langle,\rangle$ is the pairing $\H^n(M;\k^*) \times \H_n(M;\mathbb{Z})\to \k^*$.


\subsection{The DW invariant of the $n$-torus}
The DW invariant of the $n$-torus is known to be the ground state degeneracy, which is the dimension of a Hilbert space, hence an integer. Some special cases were computed in \cite{ww, wen}.

Let $\mathbb{Z}_d$ denote the quotient ring $\mathbb{Z}/d\mathbb{Z}$ and $M_n(\mathbb{Z}_d)$ the ring of $d \times d$ matrices with entries in $\mathbb{Z}_d.$ Fix a $d$-th primitive root $\xi$ of 1 and define $$N_n(d):=\frac{\sum_{A\in M_n(\mathbb{Z}_d)}\xi^{\det A}}{d^n}.$$

\begin{lemma}\label{lemma}
The function $N_n(d)$ is integer-valued and is multiplicative on $d$, that is, if $d=d_1d_2$ with $(d_1,d_2)=1$, then 
$N_n(d)=N_n(d_1)N_n(d_2)$. Moreover, for a prime $p,$ 
$$N_n(p^m)=\sum_{i=1}^mp^{m(n-2)}p^{(m-i)(n-2)(n-1)}N_{n-1}(p^i)(p^{ni}-p^{n(i-1)})+p^{m(n-2)n}.$$ 
\end{lemma}
\begin{proof}
Take $A=(a_{ij})\in M_n(\mathbb{Z}_d).$ Then
$\det A=a_{11}A_{11}+\cdots+a_{1n}A_{1n}$ where $A_{ij}$ is the algebraic cofactor of $a_{ij}$ and thus
\begin{eqnarray*}
\sum_{A\in M_n(\mathbb{Z}_d)}\xi^{\det A}&=&\sum_{a_{11}=0}^{d-1}\xi^{a_{11}A_{11}}\cdots\sum_{a_{1n}=0}^{d-1}\xi^{a_{1n}A_{1n}}\\
&=&d^n\#\{B\in M_{(n-1)\times n}(\mathbb{Z}_d)|\text{all }(n-1)\text{-minors of }B\text{ are }0\}.
\end{eqnarray*}
Hence $N_n(d)=\#\{B\in M_{(n-1)\times n}(\mathbb{Z}_d)|\text{all }(n-1)\text{-minors of }B\text{ are }0\}.$

Assume $B=(b_{ij})\in M_{(n-1)\times n}(\mathbb{Z}_d)$ is such a matrix all of whose $(n-1)$-minors are $0.$ Define $\text{ord}(b_{11},\dots,b_{1n})$ to be the smallest integer $r$ such that $d|rb_{11}, \ \dots, \ d|rb_{1n}.$ Clearly, $\text{ord}(b_{11},\dots,b_{1n})|d$. Now suppose $d=p^m$ where $p$ is prime.
If $\text{ord}(b_{11},\dots,b_{1n})=p^i$, then $p^ib_{11}=p^m\widetilde{b_{11}},  \ \dots, \ p^ib_{1n}=p^m\widetilde{b_{1n}}$. For $i\geq1$, if $p|\widetilde{b_{11}}, \ \dots, \ p|\widetilde{b_{1n}}$, then $\text{ord}(b_{11},\dots,b_{1n})\leq p^{i-1}$, contradiction.
So we may assume, without loss of generality, that $p \nmid \widetilde{b_{11}}.$ In this case, the matrix $P:=\left(\begin{array}{cccc}\overline{\widetilde{b_{11}}}&\overline{\widetilde{b_{12}}}&\ldots&\overline{\widetilde{b_{1n}}}\\
\overline{0}&\overline{1}&\ldots&\overline{0}\\
\vdots&\vdots&\ddots&\vdots\\
\overline{0}&\overline{0}&\ldots&\overline{1}\end{array}\right)$ is invertible in $M_n(\mathbb{Z}_{p^m}).$ Thus obviously,
$(\overline{b_{11}},\dots,\overline{b_{1n}})P^{-1}=(\overline{p^{m-i}},\overline{0},\dots,\overline{0})$. Assume $(\overline{b_{i1}},\dots,\overline{b_{in}})P^{-1}=(\overline{b_{i1}'},\dots,\overline{b_{in}'})$ for $i=2,\dots,n-1$. 
Since all $(n-1)$-minors of $B$ are 0, all $(n-1)$-minors of $\left(\begin{array}{cccc}\overline{p^{m-i}}&\overline{0}&\ldots&\overline{0}\\
\overline{b_{21}'}&\overline{b_{22}'}&\ldots&\overline{b_{2n}'}\\
\vdots&\vdots&\ddots&\vdots\\
\overline{b_{n-1,1}'}&\overline{b_{n-1,2}'}&\ldots&\overline{b_{n-1,n}'}\end{array}\right)$ are 0. Hence all $(n-2)$-minors of $\left(\begin{array}{ccc}\overline{b_{22}'}&\ldots&\overline{b_{2n}'}\\
\vdots&\ddots&\vdots\\
\overline{b_{n-1,2}'}&\ldots&\overline{b_{n-1,n}'}\end{array}\right)$ are 0 modulo $p^i$.
So we have
$$N_n(p^m)=\sum_{i=1}^mp^{m(n-2)}p^{(m-i)(n-2)(n-1)}N_{n-1}(p^i)(p^{ni}-p^{n(i-1)})+p^{m(n-2)n}.$$

Denote $S_n(d)=\{B\in M_{(n-1)\times n}(\mathbb{Z}_d)|\text{all }(n-1)\text{-minors of }B\text{ are }0\}$.
Then $B\mapsto(B\mod d_1,B\mod d_2)$ defines a map from $S_n(d_1d_2)$ to $S_n(d_1)\times S_n(d_2)$.
If $(d_1,d_2)=1$, then this map is clearly injective and surjective by the Chinese Remainder Theorem.
Hence $N_n(d_1d_2)=N_n(d_1)N_n(d_2)$.
So if $d=p_1^{m_1}\cdots p_r^{m_r}$ where $p_1,\dots,p_r$ are distinct primes, then $N_n(d)=N_n(p_1^{m_1})\cdots N_n(p_r^{m_r})$.
\end{proof}
\begin{theorem}\label{mt}
The Dijkgraaf-Witten invariant of the $n$-torus $T^n$ for a finite abelian group $G$ is 
\begin{equation}\label{dwtorus}
Z^{[\o]}(T^n)=\frac{1}{|G|}\sum_{f_1,\dots,f_n\in G}\frac{\prod_{\sigma\in A_n}\o(f_{\sigma(1)},\dots,f_{\sigma(n)})}{\prod_{\sigma\in S_n\setminus A_n}\o(f_{\sigma(1)},\dots,f_{\sigma(n)})}.
\end{equation}
Let $G=\mathbb{Z}_{m_1}\times\cdots\times\mathbb{Z}_{m_l}$ where $m_i|m_{i+1}$ for $1\le i\le l-1$.
If $l<n$, then $Z^{[\o]}(T^n)=|G|^{n-1}$.
If $l=n$, then $Z^{[\o]}(T^n)=\frac{|G|^{n-1}}{d^{n(n-1)}}N_n(d)$ where $d=\frac{m_1}{(m_1,a_{1\cdots n})}$.
If $l>n$, then
$$Z^{[\o]}(T^n)=\frac{1}{|G|}\sum_{A}\prod_{1\leq r_1<\cdots<r_n\leq l}\zeta_{m_{r_1}}^{a_{r_1\cdots r_n}\det A\left(\begin{array}{ccc}1&\ldots&n\\r_1&\ldots&r_n\end{array}\right)}$$ where 
$A=(\alpha_{ij})_{n \times l}$ and $0\leq \alpha_{ij}<m_j$ for $1\leq i\leq n$.
\end{theorem}
\begin{proof}
The $n$-torus $T^n$ is obtained by gluing parallel edges of an $n$-dimensional cube. The cube can be subdivided into $n!$ $n$-simplexes such that each $n$-simplex has $n$ successive edges in common with the cube.  
If we label the remaining $n$ edges of the cube after gluing by $f_1,\dots,f_n\in G$, then each $n$-simplex is uniquely determined by an permutation $(f_1',\dots,f_n')$ of $(f_1,\dots,f_n)$.
The fundamental class $[T^n]$ is represented by an $n$-chain $\sigma:\Delta^n\to T^n$ where $\sigma$ is the sum of those $n$-simplexes with the sign of which is positive if the permutation is even and negative otherwise.
By \cite[p89]{hatcher} we may identify $\H^n(BG;\k^*)$ and $\H^n(G;\k^*).$ Then when $\phi$ runs over all group homomorphisms from $\pi_1(T^n)$ to $G$, we have
$$\o((f_{\phi})_*([T^n]))=\frac{\prod_{\sigma\in A_n}\o(f_{\sigma(1)},\dots,f_{\sigma(n)})}{\prod_{\sigma\in S_n\setminus A_n}\o(f_{\sigma(1)},\dots,f_{\sigma(n)})}$$ 
where $f_1,\dots,f_n$ run over $G$.
Hence \eqref{dwtorus} holds.

Now let $f_i=(\alpha_{i1},\dots,\alpha_{il})$ where $0\le \alpha_{ij}<m_j$ for $1\le i\le n$. 
Recall that 
$$\o(f_1,\dots,f_n)=\prod_{k=1}^{n}\prod_{\begin{array}{ccc}1\leq r_{1}<\cdots<r_{k}\leq l\\\lambda_1+\cdots+\lambda_k=n,\lambda_1\text{ odd}\\\lambda_i\ge1\text{ for }1\le i\le k\end{array}}\zeta_{m_{r_1}}^{(-1)^{\sum\limits_{1\leq i<j\leq k}\lambda_{i}\lambda_{j}}\eta_{r_{1},[a_{1},b_{1}]}^{\alpha}\cdots\eta_{r_{k},[a_{k},b_{k}]}^{\alpha}a_{r_1^{\lambda_1}\cdots r_k^{\lambda_k}}}$$
where $a_{k}=1, \ b_{k}=\lambda_{k}, \ \dots, \ a_{1}=\lambda_{2}+\cdots+\lambda_{k}+1, \ b_{1}=\lambda_{1}+\cdots+\lambda_{k}=n$.

\begin{eqnarray*}
Z^{[\o]}(T^n)&=&\frac{1}{|G|}\sum_{f_1,\dots,f_n\in G}\frac{\prod_{\sigma\in A_n}\o(f_{\sigma(1)},\dots,f_{\sigma(n)})}{\prod_{\sigma\in S_n\setminus A_n}\o(f_{\sigma(1)},\dots,f_{\sigma(n)})}\\
&=&\frac{1}{|G|}\sum_{f_1,\dots,f_n\in G}\prod_{k=1}^{n}\prod_{\begin{array}{ccc}1\leq r_{1}<\cdots<r_{k}\leq l\\\lambda_1+\cdots+\lambda_k=n,\lambda_1\text{ odd}\\\lambda_i\ge1\text{ for }1\le i\le k\end{array}}\\
&&\quad \frac{\prod_{\sigma\in A_n}\zeta_{m_{r_1}}^{(-1)^{\sum\limits_{1\leq i<j\leq k}\lambda_{i}\lambda_{j}}\eta_{r_{1},[\sigma(a_{1}),\sigma(b_{1})]}^{\alpha}\cdots\eta_{r_{k},[\sigma(a_{k}),\sigma(b_{k})]}^{\alpha}a_{r_1^{\lambda_1}\cdots r_k^{\lambda_k}}}}{\prod_{\sigma\in S_n\setminus A_n}\zeta_{m_{r_1}}^{(-1)^{\sum\limits_{1\leq i<j\leq k}\lambda_{i}\lambda_{j}}\eta_{r_{1},[\sigma(a_{1}),\sigma(b_{1})]}^{\alpha}\cdots\eta_{r_{k},[\sigma(a_{k}),\sigma(b_{k})]}^{\alpha}a_{r_1^{\lambda_1}\cdots r_k^{\lambda_k}}}}
\end{eqnarray*}

If $(\lambda_1,\dots,\lambda_k)$ is a partition of $n$ and $\lambda_i>1$ for some $i\in\{1,\dots,k\}$, then $b_i>a_i$ and $\eta_{r_i,[a_i,b_i]}^{\alpha}$ contains the term $[\frac{\alpha_{b_i,r_i}+\alpha_{b_i-1,r_i}}{m_{r_i}}]$. 
Let $\tau$ be the transposition $(b_i,b_i-1)$, then 
$$\zeta_{m_{r_1}}^{(-1)^{\sum\limits_{1\leq i<j\leq k}\lambda_{i}\lambda_{j}}\eta_{r_{1},[a_{1},b_{1}]}^{\alpha}\cdots\eta_{r_{k},[a_{k},b_{k}]}^{\alpha}a_{r_1^{\lambda_1}\cdots r_k^{\lambda_k}}}
=\zeta_{m_{r_1}}^{(-1)^{\sum\limits_{1\leq i<j\leq k}\lambda_{i}\lambda_{j}}\eta_{r_{1},[\tau(a_{1}),\tau(b_{1})]}^{\alpha}\cdots\eta_{r_{k},[\tau(a_{k}),\tau(b_{k})]}^{\alpha}a_{r_1^{\lambda_1}\cdots r_k^{\lambda_k}}}.$$
Hence 
$$\frac{\prod_{\sigma\in A_n}\zeta_{m_{r_1}}^{(-1)^{\sum\limits_{1\leq i<j\leq k}\lambda_{i}\lambda_{j}}\eta_{r_{1},[\sigma(a_{1}),\sigma(b_{1})]}^{\alpha}\cdots\eta_{r_{k},[\sigma(a_{k}),\sigma(b_{k})]}^{\alpha}a_{r_1^{\lambda_1}\cdots r_k^{\lambda_k}}}}{\prod_{\sigma\in S_n\setminus A_n}\zeta_{m_{r_1}}^{(-1)^{\sum\limits_{1\leq i<j\leq k}\lambda_{i}\lambda_{j}}\eta_{r_{1},[\sigma(a_{1}),\sigma(b_{1})]}^{\alpha}\cdots\eta_{r_{k},[\sigma(a_{k}),\sigma(b_{k})]}^{\alpha}a_{r_1^{\lambda_1}\cdots r_k^{\lambda_k}}}}=1.$$
Therefore,
$$Z^{[\o]}(T^n)=\frac{1}{|G|}\sum_{f_1,\dots,f_n\in G}\prod_{1\leq r_{1}<\cdots<r_{n}\leq l}\frac{\prod_{\sigma\in A_n}\zeta_{m_{r_1}}^{(-1)^{\frac{n(n-1)}{2}}\alpha_{\sigma(n),r_1}\cdots\alpha_{\sigma(1),r_n}a_{r_1\cdots r_n}}}{\prod_{\sigma\in S_n\setminus A_n}\zeta_{m_{r_1}}^{(-1)^{\frac{n(n-1)}{2}}\alpha_{\sigma(n),r_1}\cdots\alpha_{\sigma(1),r_n}a_{r_1\cdots r_n}}}.
$$

Hence if $l<n$, then each summand of \eqref{dwtorus} is 1 and $Z^{[\o]}(T^n)=|G|^{n-1}$.
If $l=n$, then equation \eqref{dwtorus} becomes  
\begin{eqnarray*}
Z^{[\o]}(T^n)&=&\frac{1}{|G|}\sum_{A}(\zeta_{m_1}^{a_{1\cdots n}})^{\sum_{\sigma\in S_n}(-1)^{\text{sign of }\sigma}(-1)^{\frac{n(n-1)}{2}}\alpha_{\sigma(n)1}\cdots\alpha_{\sigma(1)n}}\\
&=&\frac{1}{|G|}\sum_{A}(\zeta_{m_1}^{a_{1\cdots n}})^{\sum_{\sigma\in S_n}(-1)^{\text{sign of }\sigma}\alpha_{\sigma(1)1}\cdots\alpha_{\sigma(n)n}}\\
&=&\frac{1}{|G|}\sum_{A}(\zeta_{m_1}^{a_{1\cdots n}})^{\det A}
\end{eqnarray*}
where $A=(\alpha_{ij})_{n \times l}$ and $0\leq \alpha_{ij}<m_j$ for $1\leq i\leq n$.
Denote $\xi=\zeta_{m_1}^{a_{1\cdots n}}$, then $\xi$ is a $d$-th primitive root of 1 where $d=\frac{m_1}{(m_1,a_{1\cdots n})}$. In this situation,
\begin{eqnarray*}
Z^{[\o]}(T^n)&=&\frac{1}{|G|}\sum_{A}\xi^{\det A}\\
&=&\frac{1}{|G|}(\frac{m_1}{d}\cdots\frac{m_n}{d})^n\sum_{A\in M_n(\mathbb{Z}_d)}\xi^{\det A}\\
&=&\frac{|G|^{n-1}}{d^{n^2}}\sum_{A\in M_n(\mathbb{Z}_d)}\xi^{\det A}.
\end{eqnarray*}
If $l>n$, the formula is similarly derived as the case $l=n$.
Hence Lemma \ref{lemma} completes the proof.
\end{proof}

\subsection{The DW invariant of $T^2$ and projective representations}
In \cite{turaev} Turaev observed the connection between DW invariants of surfaces and projective representations of finite groups. In case of $T^2,$ our Theorem \ref{mt} 
recovers some partial results of Turaev. Moreover, with a help of our explicit formula of $2$-cocycles, we are able to derive a formula for the dimension of an arbitrary projective representation of finite abelian groups. This is of independent interest on the one hand, and helps to improve some formulas in \cite{turaev} on the other hand.

Now let $G=\mathbb{Z}_{m_1}\times\cdots\times\mathbb{Z}_{m_l}$ with $m_1|m_2|\cdots|m_l$ and let $\o$ be a $2$-cocycle on $G$.
In case $k=2$, \eqref{barcocycleabelian} becomes
$$\o(g_1^{i_1}\cdots g_l^{i_l},g_1^{j_1}\cdots g_l^{j_l})=\prod_{1\leq r<s\leq l}\zeta_{m_r}^{-a_{rs}i_sj_r}.$$
Let $G_0$ be the set of all $\o$-regular elements in $G$, i.e.,
$$G_0=\{x\in G|\o(x,y)=\o(y,x)\text{ for all }y\in G\}.$$ It is well known that the number of irreducible $\o$-representations of $G$ is $|G_0|$ and all irreducible $\o$-representations of $G$ share a common dimension $d=\sqrt{\frac{|G|}{|G_0|}},$ see \cite{frucht, kar}. 

In the following we derive the formula of $|G_0|,$ hence of $d$ as well, in terms of the data $(a_{rs})_{1\leq r<s\leq l}$ of the given $2$-cocycle $\o.$ Consider the following antisymmetric $l \times l$-matrix
$$B=\left(\begin{array}{cccc}0&b_{12}&\ldots&b_{1l}\\-b_{12}&0&\ldots&b_{2l}\\\vdots&\vdots&\ddots&\vdots\\-b_{1l}&-b_{2l}&\ldots&0\end{array}\right)$$
where $b_{ij}=\frac{m_l}{m_i}a_{ij}$. Assume that the invariant factors of $B$ are $\lambda_1, \ \dots, \ \lambda_k$ with $\lambda_1|\lambda_2|\cdots|\lambda_k.$

\begin{proposition}
Keep the above notations. Then we have \[ |G_0|=\frac{|G|}{\frac{m_l}{(m_l,\lambda_1)}\cdots\frac{m_l}{(m_l,\lambda_k)}}, \quad \quad d=\sqrt{\frac{m_l}{(m_l,\lambda_1)}\cdots\frac{m_l}{(m_l,\lambda_k)}}. \]
\end{proposition}
\begin{proof}
By direct computations, we have
\begin{eqnarray*}
|G_0|&=&\#\{(i_1,\dots,i_l)|0\le i_r<m_r\text{ for }1\le r\le l,\prod_{1\leq r<s\leq l}\zeta_{m_r}^{-a_{rs}i_sj_r}=\prod_{1\leq r<s\leq l}\zeta_{m_r}^{-a_{rs}i_rj_s}\\
&&\text{ for any }(j_1,\dots,j_l)\text{ where }0\le j_r<m_r\text{ for }1\le r\le l\}\\
&=&\frac{1}{\frac{m_l}{m_1}\cdots\frac{m_l}{m_{l-1}}}\#\{(i_1,\dots,i_l)|0\le i_r<m_l\text{ for }1\le r\le l,\left(\begin{array}{ccc}i_1&\ldots&i_l\end{array}\right)B\left(\begin{array}{ccc}j_1\\\vdots\\j_l\end{array}\right)\\
&&\equiv0\mod m_l\text{ for any }(j_1,\dots,j_l)\text{ where }0\le j_r<m_l\text{ for }1\le r\le l\}\\
&=&\frac{1}{\frac{m_l}{m_1}\cdots\frac{m_l}{m_{l-1}}}\#\{(i_1,\dots,i_l)|0\le i_r<m_l\text{ for }1\le r\le l,\lambda_ri_r\equiv0\mod m_l\text{ for }1\le r\le k\}\\
&=&\frac{1}{\frac{m_l}{m_1}\cdots\frac{m_l}{m_{l-1}}}(m_l,\lambda_1)\cdots(m_l,\lambda_k)m_l^{l-k}\\
&=&\frac{|G|}{\frac{m_l}{(m_l,\lambda_1)}\cdots\frac{m_l}{(m_l,\lambda_k)}}
\end{eqnarray*} and \[ d=\sqrt{\frac{|G|}{|G_0|}}=\sqrt{\frac{m_l}{(m_l,\lambda_1)}\cdots\frac{m_l}{(m_l,\lambda_k)}}. \]
\end{proof}

We recover a result of Turaev \cite{turaev} in the following
\begin{corollary}
Keep the previous assumptions and notations. We have $$Z^{[\o]}(T^2)=|G_0|=\frac{|G|}{\frac{m_l}{(m_l,\lambda_1)}\cdots\frac{m_l}{(m_l,\lambda_k)}}.$$
\end{corollary}
\begin{proof}
If $l\ge2$, then
\begin{eqnarray*}
Z^{[\o]}(T^2)&=&\sum_A\prod_{1\leq r_1<r_2\leq l}\zeta_{m_{r_1}}^{a_{r_1r_2}\det A\left(\begin{array}{cc}1&2\\r_1&r_2\end{array}\right)}\\
&=&\frac{1}{(\frac{m_l}{m_1})^2\cdots(\frac{m_l}{m_{l-1}})^2}\sum_{0\leq\alpha_{ij}<m_l}\zeta_{m_l}^{\left(\begin{array}{ccc}\alpha_{11}&\ldots&\alpha_{1l}\end{array}\right)B\left(\begin{array}{ccc}\alpha_{21}\\\vdots\\\alpha_{2l}\end{array}\right)}
\end{eqnarray*}
where $A=\left(\begin{array}{ccc}\alpha_{11}&\ldots&\alpha_{1l}\\\alpha_{21}&\ldots&\alpha_{2l}\end{array}\right)$ and $0\leq\alpha_{ij}<m_j$.

By assumption, there exist two invertible integral matrices $P, \ Q \in GL_l(\mathbb{Z})$ such that 
$$B=P\left(\begin{array}{cccccc}\lambda_1&&&&&\\&\ddots&&&&\\&&\lambda_k&&&\\&&&0&&\\&&&&\ddots&\\&&&&&0\end{array}\right)Q.$$
Note that the images of $P$ and $Q$ in $M_l(\mathbb{Z}_{m_l})$ are also invertible. Hence 
\begin{eqnarray*}
&&\sum_{0\leq\alpha_{ij}<m_l}\zeta_{m_l}^{\left(\begin{array}{ccc}\alpha_{11}&\ldots&\alpha_{1l}\end{array}\right)B\left(\begin{array}{ccc}\alpha_{21}\\\vdots\\\alpha_{2l}\end{array}\right)}\\
&=&\sum_{0\leq\alpha_{ij}<m_l}\zeta_{m_l}^{\left(\begin{array}{ccc}\alpha_{11}&\ldots&\alpha_{1l}\end{array}\right)\left(\begin{array}{cccccc}\lambda_1&&&&&\\&\ddots&&&&\\&&\lambda_k&&&\\&&&0&&\\&&&&\ddots&\\&&&&&0\end{array}\right)\left(\begin{array}{ccc}\alpha_{21}\\\vdots\\\alpha_{2l}\end{array}\right)}\\
&=&m_l^{2(l-k)}\sum_{\alpha_{11}=0}^{m_l-1}\sum_{\alpha_{21}=0}^{m_l-1}\zeta_{m_l}^{\alpha_{11}\lambda_1\alpha_{21}}\cdots\sum_{\alpha_{1k}=0}^{m_l-1}\sum_{\alpha_{2k}=0}^{m_l-1}\zeta_{m_l}^{\alpha_{1r}\lambda_r\alpha_{2r}}
\end{eqnarray*}
It is well known that if $\xi^m=1$, then 
$$\sum_{i=0}^{m-1}\xi^{id}=\left\{\begin{array}{ll}m, &\text{if }\xi^d=1;\\0,&\text{if }\xi^d\ne1.\end{array}\right.$$
Thus we have 
\begin{eqnarray*}
Z^{[\o]}(T^2)&=&\frac{1}{|G|}\frac{1}{(\frac{m_l}{m_1})^2\cdots(\frac{m_l}{m_{l-1}})^2}m_l^{2(l-k)}m_l^k(m_l,\lambda_1)\cdots(m_l,\lambda_k) =\frac{|G|}{\frac{m_l}{(m_l,\lambda_1)}\cdots\frac{m_l}{(m_l,\lambda_k)}}.
\end{eqnarray*}

If $l=1$, $Z^{[\o]}(T^2)=|G|=|G_0|$. The conclusion also holds.
\end{proof}

\section*{Acknowledgment}
The authors are grateful to J. C. Wang for bringing the works \cite{spt, ww, wen} to their attention, in connection with Subsection 2.5 and Subsection 4.2.

\end{document}